\documentclass[12pt]{amsart}
\usepackage{cite}
\usepackage{amsmath}
\usepackage{amsmath,amsthm,amsfonts,amssymb}
\usepackage{tikz-qtree}
\usepackage[utf8]{inputenc}
\usepackage[all,cmtip]{xy}
\usepackage[english]{babel}
\usepackage[T1]{fontenc}
\usepackage{amsmath}
\usepackage{xcolor}
\usepackage{ytableau}
\usepackage{amsfonts}
\usepackage{amssymb}
\usepackage{graphicx}
\usepackage[margin=.8in]{geometry}

\newcommand{\pocol}{{\bf $\text{ACol}$}}
\newcommand{\pohag}{{\bf $\text{ACol}_\bullet$}}
\newcommand{\cons}{\mathfrak{c}}
\newcommand{\acol}{\text{ACol}(C_n)}
\newcommand{\adm}{\text{ACol}(C_n)}
\newcommand{\bull}{\text{ACol}_\bullet(C_n)}
\newtheorem{prop}{Proposition}[section]
\newtheorem{thm}[prop]{Theorem}
\newtheorem{defin}[prop]{Definition}
\newtheorem{lemm}[prop]{Lemma}
\newtheorem{cor}[prop]{Corollary}
\newtheorem{remark}[prop]{Remark}
\newtheorem{exo}[prop]{Example}

\newtheorem{manualtheoreminner}{Theorem}
\newenvironment{manualtheorem}[1]{%
	\renewcommand\themanualtheoreminner{#1}%
	\manualtheoreminner
}{\endmanualtheoreminner}

\newtheorem{manualtheoreminnerr}{Corollary}
\newenvironment{manualcorollary}[1]{%
	\renewcommand\themanualtheoreminnerr{#1}%
	\manualtheoreminnerr
}{\endmanualtheoreminnerr}

\newtheorem{manualtheoreminnerrr}{2-column Lemmata}
\newenvironment{manuallemmata}[1]{%
	\renewcommand\themanualtheoreminnerrr{#1}%
	\manualtheoreminnerrr
}{\endmanualtheoreminnerrr}

\title{$C-$trees and a coherent presentation for the plactic monoid of type $C$}
\date{}
\author{Uran Meha}
\begin{document}
	
	\begin{abstract}
		In this article we introduce the $\mathbb{N}-$decorated plactic monoid of type $C$, denoted $Pl^\mathbb{N}(C_n)$, via a finite convergent presentation {\bf $\text{ACol}$}, with generating set $\adm$ consisting of admissible columns, and an element $\epsilon$. By Squier's coherent completion theorem, this presentation is extended into a coherent presentation by identifying a family of generating confluences, i.e. generating $3-$cells. Here the generating $3-$cells are critical branchings on words of length $3$. We adapt the notions of crystal structure to $\adm^*$, and show that the shape of $3-$cells is preserved by the action of Kashiwara operators. Thus we reduce the study of the coherent presentation to only describing the generating $3-$cells whose source is a word of highest weight. We then introduce combinatorial objects called $C-$trees which parameterize the words of highest weight in $\adm^*$. The $C-$trees allow for simplifying calculations with the insertion algorithm in type $C$, as introduced in by Lecouvey, and we prove that the generating $3-$cells in {\bf $\text{ACol}$} are of shape at most $(4,3)$. As a consequence, we show that the column presentation of $Pl(C_n)$, as introduced by Hage, has generating $3-$cells of shape at most $(4,3)$. This contrasts the situation in type $A$, where the $3-$cells in the column presentation of $Pl(A_n)$  are of shape at most $(3,3)$.
	\end{abstract}
	\maketitle
	\section{Introduction} \noindent In \cite{schensted1961longest}, Schensted in his study of determining the lengths of the longest decreasing and increasing subwords of a word, introduced the \textit{insertion algorithm} of a Young tableaux into another. This algorithm associates a semistandard Young tableau $P(w)$ to each word $w$ in the alphabet $A_n=\{1,2,\dots,n\}$. He also proved the \textit{Robinson-Schensted (RS) correspondence}, which establishes a bijection between the permutations of $A_n$ and pairs of standard Young tableaux $(P,Q)$ of the same shape. In \cite{knuth1970permutations}, Knuth describes the generating relations in $A_n^*$
	\begin{equation}\label{eq:knuth relations}
	\begin{array}{cc}
	xzy \sim zxy & \text{ if } x\leq y < z,\\
	yzx \sim yxz & \text{ if } x< y \leq z,
	\end{array}
	\end{equation}
	so that we have $P(w)=P(w')$ if and only if $w'$ is obtained by successively applying relations as in \eqref{eq:knuth relations} to $w$. Moreover, Knuth generalizes the RS correspondence to the \textit{Robinson-Schensted-Knuth (RSK) correspondence}, which establishes a bijection between words in $A_n^*$ and pairs of Young tableaux $(P,Q)$ of the same shape, where $P=P(w)$ is semistandard, and $Q$ is standard. In \cite{lascoux1981monoide}, Lascoux and Sch\"utzenberger define the \textit{plactic monoid} as $Pl(A_n)=(A_n^*/\equiv)$, where '$\equiv$' is the congruence on $A_n^*$ generated by the relations in \eqref{eq:knuth relations}. They used their construction to give a proof of the Littlewood-Richardson rule, which describes the decomposition of tensor product of irreducible $\mathfrak{gl}_n$-modules. In \cite{cain2019crystal} a \textit{column presentation} of $Pl(A_n)$ was introduced. In \cite{hage2017knuth}, using rewriting theory, the authors introduce extend this into a coherent presentation. This presentation has generators  the columns in $A_n^*$, and the generating $2-$cells are rewriting rules of the form $c_1c_2\Longrightarrow (c_1\leftarrow c_2)$, where $(c_1\leftarrow c_2)$ is the output of the Schensted insertion of $c_2$ into $c_1$. They show that this presentation is finite and convergent, and moreover they explicitly describe the generating $3-$cells, in other words they explicitly describe a finite coherent presentation of $Pl(A_n)$. In \cite{lebed2016plactic}, the column (and row) insertion are shown to induce certain braidings, which are then used to study $Pl(A_n)$ and calculate its cohomology.

	\noindent Kashiwara's introduction of the crystal bases for classical Lie algebras $\mathfrak{g}$ in \cite{kashiwara1990crystalizing} gives rise to crystal monoids associated to $\mathfrak{g}$. In type $A$, this crystal monoid is isomorphic to $Pl(A_n)$, meaning that there exists a presentation of this crystal monoid by generators and relations. In \cite{littelmann1996plactic}, using his own path model, Littelmann introduces a plactic algebra $\mathbb{Z}\mathcal{P}(\mathfrak{g})$ for any symmetrizable Kac-moody algebra $\mathfrak{g}$ over $\mathbb{C}$. In his Theorem B, he defines a presentation of $\mathbb{Z}\mathcal{P}(\mathfrak{g})$, which in type $A$ agrees with the column presentation.
	\smallskip
	
	\noindent Denote the plactic monoid of type $C$ by $Pl(C_n)$. Many of the relevant notions in $Pl(A_n)$ exist for $Pl(C_n)$ as well. Symplectic tableau were introduced in \cite{de1979symplectic}. In \cite{lascoux1995crystal} a presentation of $Pl(C_n)$ by generators and relations is given in the style of Knuth. In \cite{lecouvey2002schensted} this is proven to indeed be a presentation of $Pl(C_n)$. Moreover, Lecouvey defines a Schensted-like column insertion algorithm for symplectic tableaux, and uses it to establish a symplectc Robinson-Schensted correspondence. In \cite{hage2015finite} a \textit{column presentation} for $Pl(C_n)$ is given, to be denoted by {\bf $\text{ACol}_\bullet$} in this paper. Its generators are the non-empty admissible columns, and the relations are generated by rewriting rules of the form $c_1c_2\Longrightarrow (c_1\leftarrow c_2)$, where $(c_1\leftarrow c_2)$ denotes the column insertion of $c_2$ into $c_1$. This presentation is shown to be finite and convergent. 
	
	\smallskip
	\noindent Squier's Theorem from \cite{squier1994finiteness}, allows us to extend this presentation into a coherent one by identifying the confluence diagrams of a family of generating confluences. In our case, this family consists of the critical branchings of {\bf $\text{ACol}_\bullet$}.  The approach of using rewriting theory to explicit coherent presentation appears for example in \cite{gaussent2015coherent} for Artin monoids, in \cite{hage2017knuth} for the plactic monoid $Pl(A_n)$, and in \cite{hage2019coherence} for Chinese monoids. For $Pl(A_n)$, the critical branchings have sources of the form $c_1c_2c_3$ for $c_1,c_2,c_3$ columns in $A_n$. There are two reduction strategies of $c_1c_2c_3$ to its normal form. Each of these alternates insertions on the \textit{left pair} $c_1c_2$, and the \textit{right pair} $c_2c_3$. The \textit{leftmost reduction strategy} commences with the left pair, while the \textit{rightmost reduction strategy} commences with the right pair. In \cite{hage2017knuth}, all the possible such confluence diagrams for the column presentation of $Pl(A_n)$ have been computed, and shown to be of the form
	\begin{equation*}
	{\xymatrix{ & t'u'v \ar@{=>}[r] & t''u'''v' \ar@{=>}[rd]\\ tuv \ar@{=>}[ru]^{\alpha_{tu}v} \ar@{=>}[rd]_{t\alpha_{uv}}&  &  & t_0u_0v_0 \\ & t u_1 v_1 \ar@{=>}[r] &t_1u_2v_1 \ar@{=>}[ur]}}
	\end{equation*}
	where $t_0u_0v_0$ is the normal form of $tuv$, and we allow for some of the arrows to be identities. We say that the confluence diagrams are of shape $(3,3)$.
	
	\smallskip
	\noindent In this article we prove a similar result for the column presentation {\bf $\text{ACol}_\bullet$} of $Pl(C_n)$, namely we compute the general form of the confluence diagram to be $(4,3)$. The manner in which we arrive at this result can be summarized as follows
	
	\begin{itemize}
		\item[1.] Introduce the $\mathbb{N}-$decorated plactic monoid $Pl^\mathbb{N}(C_n)$ of type $C$ via a finite convergent presentation \pocol. Its generators are $\acol = \{\text{admissible columns}\}\sqcup \{\epsilon\}$, with $\epsilon$ signifying the an 'empty column', and as a set we have $Pl^\mathbb{N}(C_n)=Pl(C_n)\times \mathbb{N}$. We then adapt the Kashiwara crystal structure from $C_n^*$ to $\acol^*$ and obtain the notion of highest weight in $\acol^*$.
		\item[2.] Show that shapes of confluence diagrams in \pocol\   are stable under the application of Kashiwara operators. This reduces the problem to only computing confluence diagrams of critical branchings of highest weight.
		\item[3.] Introduce a combinatorial tool, called the $C-$trees, which parameterize the words of highest weight in $\adm^*$.
		\item[4.] Compute the confluence diagrams for the words of $C-$trees of rank $3$.
	\end{itemize}
	\subsection{Organisation of the article}
	\noindent In Section \ref{sec:prelims} we recall notions about $Pl(C_n)$, its crystal structure, insertion, and its presentations. In particular we consider certain words in $C_n^*$ called \textit{block columns}. These are column words whose letters are consecutive in the alphabet $C_n$. We derive an admissibility condition on products of such columns, and compute certain insertions with these columns.
	\smallskip
	\noindent In Section \ref{sec:aux results} we introduce the $\mathbb{N}-$decorated plactic monoid  and denote it by $Pl^\mathbb{N}(C_n)$. This monoid is very similar to $Pl(C_n)$ with its column presentation as in \cite{hage2015finite}, but here we add a generator $\epsilon$ thought of as an 'empty column', and rewriting rules involving it. This approach is motivated by that of Lebed in \cite{lebed2016plactic}, where she uses it to define a braiding on the columns and rows of the plactic monoid of type $A$ and computes the cohomology of $Pl(A_n)$. The upshot of this approach is that the insertion of an admissible column into an admissible column can then be considered as a map $\adm^2\longrightarrow \adm^2$, where $\adm$ is the set of admissible columns including $\epsilon$. In other words, the column presentation of $Pl(C_n)$ is adapted so that it can be considered a quadratic presentation. This covers for the situation where $(c_1\leftarrow c_2)=c$ is a single column, or the empty word in $Pl(C_n)$, and thus allows us to treat the insertion and the computation of confluence diagrams of critical branchings without distinguishing many cases.
	
	\noindent Using the column presentation \pohag\ of $Pl(C_n)$ as starting point, we thus define $Pl^\mathbb{N}(C_n)$ as the monoid presented by the $2-$polygraph \pocol\ $=\left(\{\ast\},\mathbf{ACol}_1,\mathbf{ACol}_2\right)$ where $\mathbf{ACol}_1=\adm$, and $\mathbf{ACol}_2$ consists of the rewriting rules $c_1c_2\Longrightarrow(c_1\leftarrow c_2)=d_1d_2$ for $c_1c_2$ not standard. We then prove the following
	
	\begin{manualtheorem}{\ref{thm:conv poly}}
		\label{thm:conv of dec}	The $2-$polygraph \pocol\  is finite and convergent.
	\end{manualtheorem}
	
	\noindent We adapt Kashiwara's crystal structure from $C_n^*$ to $\adm^*$ and define an \textit{extended crystal congruence } $\equiv_\epsilon$. We show that, as in the classical case of $Pl(C_n)$, the congruence $\equiv_\epsilon$ is identical to the one generated by $\mathbf{ACol}_2$ in \pocol. More precisely, we prove the following.
	
	\begin{manualtheorem}{\ref{prop:crystal in plactic}}
		
		Let $w_1,w_2\in\adm^*$. We have $w_1\equiv_\epsilon w_2$ if and only if $[w_1]=[w_2]\in Pl^\mathbb{N}(C_n)$. In other words $(\adm^*/\equiv_\epsilon)=Pl^\mathbb{N}(C_n)$.
	\end{manualtheorem}
	
	\noindent We then describe a notion of a \textit{reduction strategy} for a given word $w\in \adm^*$. This is simply a sequence $s$ which describes the successive positions in $|w|$ where one can apply rewriting rules. We show that reduction strategies are preserved by the action of the Kashiwara operators. This fact can be summarized in the following diagram
	$$
	
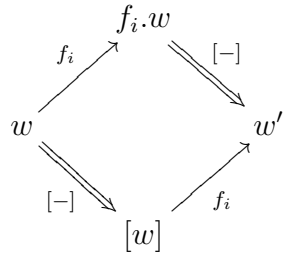
\xymatrix{& f_i.w \ar@{=>} [rd]^{[-]} \\
		w \ar@{->}[ur]^{f_i} \ar@{=>} [rd]_{[-]}& & w' \\ 
		& [w] \ar@{->}[ur]_{f_i}}
	$$
	
	\noindent We apply this to the leftmost and rightmost reduction strategies for words of the form $w=c_1c_2c_3$, i.e. the two strategies which constitute a critical branching, and characterize the corresponding confluence diagram by a pair $\text{conf}(w)=(a,b)$, where $a$ and $b$ signify the lengths of the two reduction strategies, and call $(a,b)$ the shape of the diagram. We obtain the following result.
	
	\begin{manualtheorem}{\ref{thm:confdiags}}	
		Let $t,u,v\in\adm$ and $w=tuv$. Then $\text{conf}(w)=\text{conf}(w^0)$.
	\end{manualtheorem}
	
	\noindent This result shows that it suffices to describe confluence diagrams of words $w=c_1c_2c_3\in\adm^*$ of highest weight.
	
	\smallskip
	\noindent In Section \ref{sec:Ctrees} we introduce a certain tree $\Gamma_C$ along with labels on its vertices, and call it the $C-$\textit{tree}. The purpose of it is to parametrize the words of highest weight in $\adm^*$. This is the tree 
	
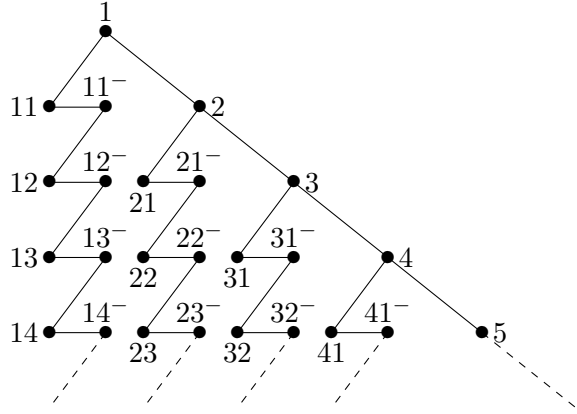
\begin{center}
		\begin{tikzpicture}
		\draw[-] (-0.25,0)--(-1,-1);
		\draw[-] (-1,-1)--(-0.25,-1);
		\draw[-] (-0.25,-1)--(-1,-2);
		\draw[-] (-1,-2)--(-0.25,-2);
		\draw[-] (-0.25,-2)--(-1,-3);
		\draw[-] (-1,-3)--(-0.25,-3);
		\draw[-] (-0.25,-3)--(-1,-4);
		\draw[-] (-1,-4)--(-0.25,-4);
		
		\draw[-] (-0.25,0)--(1,-1);
		
		\draw[-] (1,-1)--(0.25,-2);
		\draw[-] (0.25,-2)--(1,-2);
		\draw[-] (1,-2)--(0.25,-3);
		\draw[-] (0.25,-3)--(1,-3);
		\draw[-] (1,-3)--(0.25,-4);
		\draw[-] (0.25,-4)--(1,-4);
		
		\draw[-] (1,-1)--(2.25,-2);
		\draw[-] (2.25,-2)--(1.5,-3);
		\draw[-] (1.5,-3)--(2.25,-3);
		\draw[-] (2.25,-3)--(1.5,-4);
		\draw[-] (1.5,-4)--(2.25,-4);
		
		\draw[-] (2.25,-2)--(3.5,-3);
		
		\draw[-] (3.5,-3)--(2.75,-4);
		\draw[-] (2.75,-4)--(3.5,-4);
		
		\draw[-] (3.5,-3)--(4.75,-4);
		
		\node at (4.75,-4) {$\cdot$};
		\node at (-0.25,0) {$\cdot$};
		
		\draw[dashed] (-0.25,-4)--(-1,-5);
		\draw[dashed] (1,-4)--(0.25,-5);
		\draw[dashed] (2.25,-4)--(1.5,-5);
		\draw[dashed] (3.5,-4)--(2.75,-5);
		\draw[dashed] (4.75,-4)--(6,-5);
		
		\node at (-0.25,0) {$\bullet$}; \node[above] at (-0.25,0) {\small{$1$}};
		\node at (-1,-1) {$\bullet$}; \node[left] at (-1,-1) {\small{$11$}};
		\node at (-0.25,-1) {$\bullet$}; \node[above] at (-0.25,-1) {\small{$11^-$}};
		\node at (-1,-2) {$\bullet$}; \node[left] at (-1,-2) {\small{$12$}};
		\node at (-0.25,-2) {$\bullet$}; \node[above] at (-0.25,-2) {\small{$12^-$}};
		\node at (-1,-3) {$\bullet$}; \node[left] at (-1,-3) {\small{$13$}};
		\node at (-0.25,-3) {$\bullet$}; \node[above] at (-0.25,-3) {\small{$13^-$}};
		\node at (-1,-4) {$\bullet$};\node[left] at (-1,-4) {\small{$14$}};
		\node at (-0.25,-4) {$\bullet$}; \node[above] at (-0.25,-4) {\small{$14^-$}};

		\node at (1,-1) {$\bullet$}; \node[right] at (1,-1) {\small{$2$}};
		\node at (0.25,-2) {$\bullet$}; \node[below] at (0.25,-2) {\small{$21$}};
		\node at (1,-2) {$\bullet$}; \node[above] at (1,-2) {\small{$21^-$}};
		\node at (0.25,-3) {$\bullet$}; \node[below] at (0.25,-3) {\small{$22$}};
		\node at (0.25,-4) {$\bullet$}; \node[below] at (0.25,-4) {\small{$23$}};
		\node at (1,-4) {$\bullet$}; \node[above] at (1,-4) {\small{$23^-$}};
		\node at (1,-3) {$\bullet$}; \node[above] at (1,-3) {\small{$22^-$}};

		\node at (2.25,-2) {$\bullet$}; \node[right] at (2.25,-2) {\small{$3$}};
		\node at (1.5,-3) {$\bullet$}; \node[below] at (1.5,-3) {\small{$31$}};
		\node at (2.25,-3) {$\bullet$}; \node[above] at (2.25,-3) {\small{$31^-$}};
		\node at (1.5,-4) {$\bullet$}; \node[below] at (1.5,-4) {\small{$32$}};
		\node at (2.25,-4) {$\bullet$}; \node[above] at (2.25,-4) {\small{$32^-$}};
		
		\node at (3.5,-3) {$\bullet$}; \node[right] at (3.5,-3) {\small{$4$}};
		\node at (2.75,-4) {$\bullet$}; \node[below] at (2.75,-4) {\small{$41$}};
		\node at (3.5,-4) {$\bullet$}; \node[above] at (3.5,-4) {\small{$41^-$}};
		\node at (4.75,-4) {$\bullet$}; \node[right] at (4.75,-4) {\small{$5$}};
		
		
		\end{tikzpicture}
	\end{center}
	We consider labeled $C-$trees of rank $n$, which are pairs $(\Gamma_C,s)$ with $s:V(\Gamma_C)\longrightarrow \mathbb{N}$ satisfying certain \textit{finiteness}, \textit{column}, and \textit{admissibility conditions}. We denote the set consisting of such objects by $\mathcal{GT}(n)$. For this purpose, we define a \textit{reading map}
	$
	\omega:\mathcal{GT}(n)\longrightarrow \adm^*,
	$  
	which associates an admissible column to each of the 'horizontal' parts of the graph, namely to each subset $\Gamma_C,k=\{(i,j^\pm)\ |\ i+j=k\}$.
	\smallskip
	\noindent In Section \ref{sec:calculations with c trees} we compute the normal form of $\omega(T)$ for $T\in\mathcal{GT}(n)$ in the following result.
	\begin{manualtheorem}{\ref{thm:normalofC}}
		
		Let $T\in\mathcal{GT}_k(n)$, and set $q_i=q_i(T)$. Then
	$$
	[\omega(T)]=\prod_{i=0}^{k-1} \cons(q_{k-i}).
	$$
	\end{manualtheorem}
	
	\noindent Here the $q_i(T)$ are numbers associated to each of the vertical strands of the $C-$tree, and $\mathfrak{c}(a)=12\cdots a$ for $a=q_{k-i}$ are blokc columns. Moreover we show that $\omega(T)$ is a word of highest weight in $\adm^*$ for all $T\in\mathcal{GT}(n)$.
	We denote by $HW^n$ the subset of $Pl^\mathbb{N}(C_n)$ consisting of the highest weight words in $\adm^*$. We then construct a map
	$
	\mathcal{T}:HW^n\longrightarrow \mathcal{GT}(n)
	$
	and prove the following result.
	\begin{manualtheorem}{\ref{thm:kindamain}}
		The map $\mathcal{T}:HW^n\longrightarrow \mathcal{GT}(n)$ is such that
		$$
		\omega(\mathcal{T}(u))=u
		$$
		for all $u\in HW^n$. Moreover, $\mathcal{T}=\omega^{-1}$.
	\end{manualtheorem}
	\noindent Via this result, we are able to adapt the problem of computing the generating $3-$cells to the language of $C-$trees.
	
	\smallskip
	\noindent In Section \ref{sec:coherent presentations}, we use block columns and $C-$trees to perform the reductions in \pocol. In particular, we compute the lengths of the leftmost and rightmost reduction strategies for words of $C-$trees of rank $3$, and obtain the following result on the coherent presentation \pocol\ of $Pl^\mathbb{N}(C_n)$.
	\begin{manualtheorem}{\ref{thm:lastbieq}}
		The generating $3-$cells of the coherent presentation \pocol\ of $Pl^\mathbb{N}(C_n)$ are of the form
		\begin{equation}\label{eq:confidag end}
		{{\xymatrix{ & t'u'v \ar@{=>}[r] & t'u''v' \ar@{=>}[r] & t''u'''v' \ar@{=>}[rd]\\ tuv \ar@{=>}[ru]^{\alpha_{tu}v} \ar@{=>}[rd]_{t\alpha_{uv}}&  &  &  & t_0u_0v_0 \\ & t u_1 v_1 \ar@{=>}[rr] & &t_1u_2v_1 \ar@{=>}[ur]}}} 
		\end{equation}
		where we allow for some of the arrows to be the identity.
	\end{manualtheorem}

	\noindent In Section \ref{sec:aux results} we show that every reduction sequence in \pocol\ gives rise to a reduction sequence in \pohag, and using this we obtain the following result on the coherent column presentation of $Pl(C_n)$.
	
	\begin{manualcorollary}{\ref{cor:lastbieq}}
		The generating $3-$cells of the coherent presentation \pohag\ of $Pl(C_n)$ are of the form
		\begin{equation}\label{eq:confidag end}
		{{\xymatrix{ & t'u'v \ar@{=>}[r] & t'u''v' \ar@{=>}[r] & t''u'''v' \ar@{=>}[rd]\\ tuv \ar@{=>}[ru]^{\alpha_{tu}v} \ar@{=>}[rd]_{t\alpha_{uv}}&  &  &  & t_0u_0v_0 \\ & t u_1 v_1 \ar@{=>}[rr] & &t_1u_2v_1 \ar@{=>}[ur]}}} 
		\end{equation}
		where we allow for some of the arrows to be the identity.
	\end{manualcorollary}	

	\noindent We remark that the column insertion in type $C$ was coded in Python by the author, and used to compute the confluence diagrams for the column presentation of $Pl(C_n)$ for small $n$.
	
	\section{Preliminary notions}\label{sec:prelims}
	\subsection{Crystal bases and plactic monoids} \label{sub:crystalbases and monoids}
	Here we recall the notions of crystal bases, and the crystal monoid associated to a given crystal basis. This notion initially appeared in the seminal work of Kashiwara \cite{kashiwara1990crystalizing}. Kashiwara's construction reflects deep properties of the representation theory of semisimple Lie algebras. Here however, we adopt the combinatorial approach as in \cite{cain2019crystal}, which fits well with the point of view that is often taken throughout this article.
	
	\subsubsection{Words, directed labeled graphs, and crystal graphs}
	
	Let $X$ be an alphabet, and denote by $X^*$ the monoid of words over this alphabet. In particular, the empty word denoted by $\varepsilon$, is also an element in $X^*$ and is the unit.
	
	\smallskip
	\noindent A graph is denoted by $\Gamma=(V,E)$ where $V$ is the set of vertices, $E\subset V \times V $ is the set of edges. An $I-$(edge-)labeled directed graph (or a directed graph with edges labeled from $I$) is the data $(V,E,I,l)$ where $(V,E)$ is a directed graph, $I$ is an indexing set, and $l:E\longrightarrow I$ is a map. We say that the edge $e$ has label $l(e)$. A \textit{path} of length $n$ in a directed graph $(V,E)$ is a map $p:\{0,1,\dots,n\}\longrightarrow V$ such that
	$
	e_j=(p(j),p(j+1))\in E.
	$
	A morphism $\phi:\Gamma_1\longrightarrow \Gamma_2$ is a map from the vertices of $\Gamma_1$ to those of $\Gamma_2$ such that edges with label $i$ in $\Gamma_1$ are mapped to edges with label $i$ in $\Gamma_2$.
	If moreover $\phi^{-1}$ is also a morphism of labeled directed graphs, we call $\phi$ an isomorphism.
	
	\begin{defin}\label{def:crystal base}
		A crystal basis is a directed labeled graph with vertex set $X$ and label set $I$ such that
		\begin{itemize}
			\item[$i)$] For any $x\in X$ and $i\in I$, we have
			$$
			\begin{array}{c}\#\{e=(x,v)\in E\ |\ v\in V\  l(e)=i\}\leq 1\\ \#\{e=(v,x)\in E\ |\ v\in V\  l(e)=i\}\leq 1
			\end{array}
			$$
			\item[$ii)$] For any $i\in I$, there exists no infinite path with edges labeled with $i$. 
		\end{itemize}
	\end{defin}
	
	\noindent For $i\in I$ there exist two partial maps
	$
	e_i, f_i:X\longrightarrow X
	$
	by setting $e_i.y=x$ and $f_i.x=y$ if and only if
	$
	x\stackrel{i}{\longrightarrow} y
	$
	is an edge in $X$. These two maps are called the \textit{Kashiwara operators}.
	
	\smallskip
	\noindent The graph structure on $X$ is extended to all of $X^*$ by extending the Kashiwara operators to $X^*$ inductively on the length of words, denoted by $|w|$ for $w\in X^*$. More precisely we have the following.
	\begin{defin}\label{def:crystal graph}
		A crystal graph arising from a crystal basis $X$ with labels in $I$, is the directed $I-$labeled graph denoted by $\Gamma_X$ that has
		\begin{itemize}
			\item $X^*$ as vertex set,
			\item an edge $w \stackrel{i}{\longrightarrow} w'$ if and only if $e_i.w'=w$ and $f_i.w=w'$,
		\end{itemize}
		where $e_i$ and $f_i$ are defined on $X^*$ inductively on the length of words by setting for every $u,v\in X^*$
		$$
		e_i.(uv)=\left\{\begin{array}{cc}
		(e_i.u)v & \text{if }\varphi_i(u)\geq \varepsilon_i(v),\\
		u(e_i.v) & \text{if } \varphi_i(u) < \varepsilon_i(v),
		\end{array}\right.
		$$
		and
		$$
		f_i.(uv)=\left\{\begin{array}{cc}
		(f_i.u)v & \text{if }\varphi_i(u)> \varepsilon_i(v),\\
		u(f_i.v) & \text{if } \varphi_i(u) \leq \varepsilon_i(v),
		\end{array}\right.
		$$
		where $\varphi_i(u)=\#\{f_i.u,f_i^2.u,\dots\}$, and $\varepsilon_i(u)=\#\{e_i.u,e_i^2.u,\dots\}$ are also inductively defined.
		
	\end{defin}
	
	\noindent Here we describe a practical way for computing the action of the Kashiwara operators on words of $X^*$, first described by Kashiwara and Nakashima in \cite{kashiwara1991crystal} for classical Lie algebras. Consider the alphabet $M=\{+,-\}$, and let $M^*$ be the free monoid on $M$. Set
	$
	\mathcal{M} := M^*/\small{\langle (+-)=\varepsilon \rangle},
	$
	where $\varepsilon$ is the empty word in $M^*$. As a set we have
	$
	\mathcal{M}=\{-^p+^q\ |\ p,q\in\mathbb{N}\},
	$
	thus via the projection $M^* \longrightarrow \mathcal{M}$, to each element of $M^*$ corresponds an element of the form $-^p+^q$. For $i\in I$, define a map
	$
	\rho_i:X^*\longrightarrow \mathcal{M}
	$
	by setting for $x\in X$
	$$
	\rho_i(x)=\begin{cases}
	- & \text{if } \varphi_i(x)=1,\\
	+ & \text{if } \varepsilon_i(x)=1,\\
	\varepsilon & \text{otherwise}.
	\end{cases}
	$$
	and then extending $\rho_i$ to all of $X^*$ by
	\begin{equation}\label{eq:rho}
	\rho_i(x_1x_2\cdots x_k)=\rho_i(x_1)\cdots \rho_i(x_k).
	\end{equation}
	For $w=x_1x_2\cdots x_k$, the expression $\rho_i(w)=-^p+^q$ is obtained by successively canceling out the subwords $+-$ that appear on the right side of \eqref{eq:rho}. We then have the following recipe for applying the Kashiwara operators $e_i$ and $f_i$ on $w$ as follows
	
	\begin{itemize}
		\item $e_i.w$ exists if and only if $p>0$, and in that case we have
		$
		e_i.w=x_1\cdots x_{l-1}(e_i.x_l)x_{l+1}\cdots x_k,
		$
		where $l$ is such that $\rho_i(x_l)=-$ is not cancelled, and is the rightmost $-$ appearing in $-^p+^q$.
		\item $f_i.w$ exists if and only if $q>0$, and in that case we have
		$
		f_i.w=x_1\cdots x_{l-1}(f_i.x_l)x_{l+1}\cdots x_k,
		$
		where $l$ is such that $\rho_i(x_l)=+$ is not cancelled, and is the leftmost $+$ appearing in $-^p+^q$.
	\end{itemize}
	\subsection{Crystal monoid}
	Let $\Gamma_X$ be a crystal graph as in the previous section, and $w\in \Gamma_x$. Denote by $B(w)$ the connected component of $w$ in $\Gamma_X$, i.e. $B(w)$ is the full directed labeled subgraph of $\Gamma_X$ having as vertices all the words obtained by successively applying Kashiwara operators to $w$.
	
	\noindent From Definition \ref{def:crystal graph}, if $w'=f_i.w$, we have $|w|=|w'|$. Hence,
	$
	B(w)\subset \{w'\in X^*\ |\ |w'|=|w|\},
	$
	and if $X$ is finite, we see that $B(w)$ is a finite graph.
	
	\smallskip
	\noindent We define a relation $\sim_K$ on the vertices of $\Gamma_X$, namely $X^*$, by setting
	$w\sim_K w'$ if there exists a labeled graph isomorphism $\phi:B(w)\longrightarrow B(w')$ such that $\phi(w)=w'$. Let then $\equiv_K$ be the congruence on $X^*$ generated by $\sim_K$.
	
	\begin{defin}\label{def:plactic}
		The crystal monoid arising from the crystal graph $\Gamma_X$ is defined as
		$$
		Pl(X):= \left(X^*/\equiv_K\right).
		$$
		If $X$ is the crystal basis of a classical Lie algebra, we call $Pl(X)$ the \textit{crystal monoid} of type $X$.
	\end{defin}
	
	\subsection{Plactic monoid of type C}\label{sub:typeC}
	We now specify the previous discussion to the particular alphabet of type $C$.
	
	\smallskip
	\noindent The crystal basis of type $C$ and rank $n$ is
	$$
	C_n:\ 1\stackrel{1}{\longrightarrow} 2 \stackrel{2}{\longrightarrow} \dots n-1 \stackrel{n-1}{\longrightarrow} n \stackrel{n}{\longrightarrow} \overline{n} \stackrel{n-1}{\longrightarrow} \overline{n-1} \stackrel{n-2}{\longrightarrow} \dots \stackrel{2}{\longrightarrow}\overline{2} \stackrel{1}{\longrightarrow} \overline{1}.
	$$
	
	\noindent We denote the crystal monoid corresponding to this crystal basis by $Pl(C_n)$. We note that the direction of the edges in $C_n$ also endows $C_n$ with a total order $<$. As in Definition \ref{def:plactic}, the definition of $Pl(C_n)$ is graph theoretic. Lecouvey in \cite{lecouvey2002schensted} gives a presentation of $Pl(C_n)$ in terms of generators and relations as follows. The relations include those of Knuth in type $A$, and thus the monoid $Pl(C_n)$ is called the \textit{plactic monoid of type $C$}. To explicit this presentation, we first recall a few key concepts in $C_n^*$.

	\noindent $\bullet$ A word $u=x_1x_2\cdots x_k\in C_n^*$ with $x_i< x_{i+1}$ for $i=1,\dots,k-1$ is called a \emph{column}. Denote the set of all columns in $C_n^*$ by $\text{Col}(C_n)$.
	
	\smallskip
	\noindent For a column $u\in \text{Col}(C_n)$ and $z=1,\dots,n$, define
	$$
	\text{Set}_z(u):=\{x\in u\ |\ x\leq z \}\cup \{x\in u\ |\ x\geq \overline{z}\},
	$$
	and denote by $N_z(u)=\#\text{Set}_z(u)$, the cardinality of $\text{Set}_z(u)$.
	
	\smallskip
	\noindent $\bullet$  A column $u\in\text{Col}(C_n)$ is called \textit{admissible} if $u$ is non-empty and
	$
	N_z(u)\leq z$ for all $z=1,2,\dots,n$.
	Denote by $\bull$ the set of all admissible columns in $C_n^*$.
	
	\smallskip
	\noindent For a word $w\in C_n^*$, we say that $v\in C_n^*$ is a \textit{strict factor} of $w$ if there exist $w_0, w_1\in C_n^*$ such that $w=w_0vw_1$.
	
	\smallskip
	\noindent  $\bullet$ A column $u\in\text{Col}(C_n)$ is called an \textit{almost admissible} column if for every strict factor $v$ of $u$ we have $v\in\text{ACol}_\bullet(C_n)$, but $u\notin \text{ACol}_\bullet(C_n)$. One can then show that $N_z(u)\leq z+1$ for all $z=1,2,\dots,n$, see \cite{lecouvey2002schensted}.
	
	\begin{thm}(\cite{lecouvey2002schensted})
		Let $\equiv_{pl}$ be the congruence on $C_n^*$ generated by the following three families of relations
		\begin{itemize}
			\item[$R_1:$] $yzx \equiv yxz$ for $x\leq y < z$ with $z\neq \overline{x}$, and $xzy \equiv zxy$ for $x<y\leq z$ with $z\neq \overline{x}$;
			\item[$R_2:$] $y\overline{(x-1)}(x-1)\equiv yx\overline{x}$ and $x\overline{x}y\equiv\overline{x-1}(x-1)y$ for $1<x\leq n$ and $x\leq y\leq \overline{x}$;
			\item[$R_3:$] let $w$ be an almost admissible column, and let $z$ be the lowest unbarred letter such that $z,\overline{z}\in w$, and $N_z(w)=z+1$. Then $w\equiv \widetilde{w}$, where $\widetilde{w}$ is obtained from $w$ by deleting the pair $(z,\overline{z})$.
		\end{itemize}
		Then $Pl(C_n)=\left(C_n^*/\equiv_{pl}\right)$.
	\end{thm}
	
	\noindent The elements of $Pl(C_n)$ are parametrized by symplectic tableaux. These are Young diagrams, i.e. collections of left-justified boxes in rows, with the length of rows being weakly decreasing from top to bottom, and with entries from $C_n$ in its boxes satisfying certain conditions. More precisely, if $T$ is a Young diagram with entries from $C_n$ in its boxes, and $C_1,C_2,\dots,C_k$ are its columns numbered from left to right, we have $C_{i}\preceq C_{i+1}$. Here $\preceq$ is a partial order in $\bull$ defined as follows.
	
	\smallskip
	\noindent $\bullet$ Let $c=x_1\cdots x_k$, $d=y_1\cdots y_l\in \bull$. We write $c\preceq d$ if
	\begin{itemize}
		\item[-] $k\geq l$
		\item[-] $x_i\leq y_i$ for $i=1,\dots,l$
		\item[-] there exists no pair of numbers $(a,b)$ such that $1\leq a\leq b\leq n$ and that the following conditions are satisfied
		\begin{itemize}
			\item[1.] $x_p=a,y_q=b,y_r=\overline{b},y_s=\overline{a}$ or $x_p=a,x_q=b,x_r=\overline{b}, y_s=\overline{a}$ for some $1\leq p\leq q < r\leq s \leq l$,
			\item[2.] $(s-r)+(q-p)\geq b-a$.
		\end{itemize}
	\end{itemize}
	The third condition is due to \cite{kashiwara1991crystal}, and any such pair $(a,b)$ is known as an $(a,b)-$configuration. It may look cumbersome in this form, but it will shortly be expressed differently in terms of the insertion algorithm.
	
	\begin{exo}
		If
		$$
		T_1=\begin{ytableau}
		1 & 1 & 2\\
		2 & \overline{3}
		\end{ytableau}\quad \text{and}\quad 
		T_2=\begin{ytableau}
		2 & 2\\
		3  & 3\\
		\overline{3} &\overline{2} 
		\end{ytableau}
		$$
		then $T_1$ is a symplectic tableau in $Pl(C_3)$, while $T_2$ contains a $(2,3)-$configuration, hence it is not a symplectic tableau in $Pl(C_3)$, despite satisfying the first two conditions.
	\end{exo}
	
	\noindent We denote by $[-]:C_n^*\longrightarrow Pl(C_n)$ the natural projection, and for $w\in C_n^*$, we call the symplectic tableaux $[w]$ its \textit{normal form}.
	\noindent There is a reading map $g:\{\text{symplectic tableau}\}\longrightarrow C_n^*$, where for $T$ a symplectic tableau with columns $c_1,\dots,c_k$ numbered from left to right, we set $g(T)=g(c_k)\cdots g(c_1)$, where $g(c)=x_1x_2\cdots x_k\in C_n^*$ for $c$ a column with entries $x_1,x_2,\cdots,x_k$. Often we shall use a symplectic tableau $T$ and its word $g(T)$ interchangeably.
	
	\begin{remark}
		We note that throughout this article we read a tableau by reading its columns from right to left. For instance if $T_1$ is as in the previous example, its reading is $g(T_1)=(2)(1\overline{3})(12)$, where the brackets signify the columns.
	\end{remark}
	\subsection{Insertion algorithm in type $C$} (\cite{lecouvey2002schensted})
	We recall the insertion algorithm for type $C$ as introduced in \cite{lecouvey2002schensted}.
	
	\smallskip
	\noindent Let $c\in\bull$ and $x\in C_n$, and consider the word $w=cx$. Then the highest weight $w^0$ of $w$ admits one of the forms
	$$
	w^0=\begin{cases}
	12\dots p (p+1),\\
	12\dots p 1,\\
	12\dots p \overline{p}.
	\end{cases}
	$$
	The insertion $c\leftarrow x$ is defined as follows
	\begin{itemize}
		\item[1.] if $w^0=12\cdots p(p+1)$, then $(c\leftarrow x)=w\in\bull$.
		\item[2.] if $w^0=12\cdots p 1$, then $(c\leftarrow x)=yc'$, where $yc'$ is the word obtained by successively applying plactic relations of the form $R_1$ and $R_2$ on $w$ from right to left.
		\item[3.] if $w^0=12\cdots p\overline{p}$, then $(c\leftarrow x)=\widetilde{cx}$ via a plactic relation $R_3$.
	\end{itemize}
	Using the insertion of a letter into a column, we can define the insertion of a letter into a symplectic tableau.  Let $T=c_1c_2\cdots c_k$ be a symplectic tableau in $Pl(C_n)$, where $c_i$ are admissible columns such that $c_{i+1}\preceq c_i$, and let $x\in C_n$. Consider the word $w=c_kx\in Pl(C_n)$. The insertion of $x$ into $T$, denoted by $T\leftarrow x$ is a new tableau $T'$ defined as follows
	\begin{itemize}
		\item[1.] if $w^0=12\dots p(p+1)$, and $(c_k\leftarrow x)=w$, then
		$$
		(T\leftarrow x)= c_1c_2\dots c_{k-1}w.
		$$
		\item[2.] if $w^0=12\dots p 1$, and $(c_1\leftarrow x)=yc_k'$, then
		$$
		(T\leftarrow x)= (c_1c_2\dots c_{k-1} \leftarrow y)c_k'
		$$
		\item[3.] if $w^0=12\dots p \overline{p}$, and $(c_k\leftarrow x)=c_k'=y_1\cdots y_l$, then
		$$
		(T\leftarrow x)=(((c_1\dots c_{k-1} \leftarrow y_1)\leftarrow y_2)\leftarrow \cdots \leftarrow y_l).
		$$
	\end{itemize}
	\noindent Using this insertion algorithm, we can define the insertion of a symplectic tableau $T_2$ into another symplectic tableau $T_1$ by setting
	\begin{equation}\label{eq:general insertion}
	(T_1\leftarrow T_2):=((((T_1\leftarrow x_1)\leftarrow x_2)\leftarrow \cdots) \leftarrow x_k),
	\end{equation}
	where $g(T_2)=x_1\cdots x_k$.
	
	\smallskip
	\noindent This insertion algorithm is such that given $T_1,T_2\in Pl(C_n)$, we have $(T_1\leftarrow T_2)=[g(T_1)g(T_2)]$.
	
	\subsubsection{Column insertion}\label{subsec:Col insertion} Of particular interest to us is the insertion of a column into another, namely $(c\leftarrow d)$ for $c,d\in\bull$. This is defined in equation \eqref{eq:general insertion}. An important result from \cite{hage2015finite}, and also appearing in \cite{cain2019crystal} is the following.
	
	\begin{manuallemmata}{}
		\label{lem:2collem}
		Let $t,u\in\text{ACol}_\bullet(C_n)$ be such that $u\npreceq t$, i.e. $g(tu)$ is not the reading of a symplectic tableau. Then $[tu]=(t\leftarrow u)$ consists of at most two columns. Moreover, if $[tu]=t'u'$, we have $|u'|<|u|$.
	\end{manuallemmata}
	
	\noindent Note that if $n=3$, and $c_1=1$, $c_2=2$, $c_3=\overline{1}$ we have
	$$
	(c_1\leftarrow c_1) = (1)(1);\quad (c_1\leftarrow c_2) = (12);\quad (c_1\leftarrow c_3) = \varnothing,
	$$
	where the brackets signify the columns. Thus for $t,u\in\bull$ we may have that $(t\leftarrow u)$ consists of $0$, $1$, or $2$ columns. For purposes of treating the insertion in functional terms, in Section \ref{sec:aux results} we will include an element $\epsilon$ in $\acol$ to signify an 'empty column'.
	
	\smallskip
	\noindent We note here that the condition $c_2\preceq c_1$ is equivalent to $(c_1\leftarrow c_2)=c_1c_2$.
	
	\subsection{Block columns}\label{subsec:conseccols} Here we describe certain column words in $C_n^*$ which will be useful in describing words of highest weight in $\bull$, and will later on be useful in Section \ref{sec:Ctrees}.
	\begin{defin}\label{def:consecols}
		A column $c=x_1x_2\cdots x_k\in C_n^*$ is called a block column (or simply a block), if for each $i=1,\dots,k-1$, there exists $j_i\in\{1,\dots,n\}$ such that $x_{i+1}=f_{j_i}.x_i$.
	\end{defin}
	\noindent In other words, $c$ is a block column if the letters in $c$ are consecutive letters in the alphabet $C_n$.
	
	\noindent We will particularly be interested in certain words that are products of blocks. Given $a,b,c\in\{0,1,2\dots,n\}$, if $c\leq a+b\leq n$, and $b,c\neq 0$ we set
	$$
	\mathfrak{c}(a;b,c):= (a+1)(a+2)\dots(a+b)\overline{a+b}\ \overline{a+b-1}\cdots\overline{a+b-c+1}.
	$$ 
	If $b\neq c=0$, we adopt the notation
	$$
	\mathfrak{c}(a;b,c)=\mathfrak{c}(a,b):= (a+1)\cdots(a+b),
	$$
	and if $c\neq b=0$ we adopt the notation
	$$
	\mathfrak{c}(a;0,c)=\mathfrak{c}(\overline{a},c):=\overline{a}\ \overline{a-1}\cdots\overline{a-c+1}.
	$$
	It is clear that $\mathfrak{c}(a,b)$ and $\mathfrak{c}(\overline{a},c)$ are blocks, and so is $\cons(a;b,c)$ if $a+b=n$.
	
	\subsubsection{Products of block columns}\label{subsec:prodsofcon}
	
	Here investigate words that are products of blocks. In particular, we describe a criterion for checking admissibility of such words. This result will be used in Section \ref{sec:Ctrees}.

	\begin{prop} \label{prop:admissible tree} Let $c_1,\dots,c_t$ and $d_1,\dots,d_r$ be blocks in $C_n^*$ such that $c_i$ consists of unbarred letters for all $i$, $d_j$ consists of barred letters for all $j$, and such that $w=c_1\cdots c_t d_1\cdots d_r$ is a column.
		Denote by $x_i$ the largest (i.e. rightmost) letter in $c_i$, and by $\overline{y_j}$ the smallest (i.e. leftmost) letter in $d_j$. Then $w$ is admissible if and only if
		\begin{equation}\label{eq:inprop1}
		N_z(w)\leq z\quad \text{for all } z=x_i,y_j.
		\end{equation}
		
	\end{prop}
	\begin{proof}
		If $w$ is admissible, then \eqref{eq:inprop1} is satisfied.
		
		\noindent Conversely, suppose that \eqref{eq:inprop1} is satisfied. We need to check whether $w$ is admissible. Let $k\in\{1,\dots,n\}$.
		
		\smallskip
		\noindent Assume that $k\in w$ or $\overline{k}\in w$. Here we treat the case $k\in w$, while the case $\overline{k}\in w$ is completely analogous. As $k$ is unbarred, there exists some $1\leq i \leq t$ such that $k\in c_i$. Let $z=x_i$. We then have
		$$
		\text{Set}_k(w)=\text{Set}_z(w)\setminus\left(\{k+1,\dots,z\}\cup\{\overline{k+1},\dots,\overline{z}\}\right)
		$$
		By assumption, we have $\{k+1,\dots,z\}\subset \text{Set}_z(w)$, and let $Q=\{\overline{k+1},\dots,\overline{z}\}\cap \text{Set}_z(w)$. We then have
		$$
		\text{Set}_k(w)=\text{Set}_z(w)\setminus\left(\{k+1,\dots,z\}\sqcup Q\right),
		$$
		and we get
		$$
		N_k(w)=N_z(w)-(z-k)-\#Q\leq z-(z-k)-\#Q=k-\#Q\leq k,
		$$
		showing that $w$ indeed satisfies the admissibility condition for $k$. Hence $w$ is admissible.
		
		\noindent Assume now that $k,\overline{k}\notin w$. Let
		$$
		l=\max\{i\leq t\ |\ x_i < k\};\ m=\min\{j\leq r\ |\ y_j<k\}.
		$$
		We then have
		$$
		\text{Set}_k(w)=\{x\in c_1\cdots c_l\}\sqcup \{x\in d_md_{m+1}\cdots d_r\}.
		$$
		If $l$ or $m$ does not exist, then the left set, respectively the right set, is empty.
		Since $k>x_l,y_m$, we obtain
		$$
		N_k(w)\leq N_{\max(x_l,y_m)}(w)\leq \max(x_l,y_m)<k,
		$$
		which shows that $w$ satisfies the admissibility condition for $k$. Hence $w$ is admissible. This completes the proof of the proposition.

	\end{proof}
	\noindent We make here note of a computational result, which can be proved by making direct computations of the insertion.
	\begin{lemm}\label{lemm: ref for n 2}
		Let $a\leq p$. Then
		\begin{itemize}
			\item[1.] $ (\mathfrak{c}(p)\leftarrow \mathfrak{c}(a))=\mathfrak{c}(a)\mathfrak{c}(p)$
			\item[2.] $(\mathfrak{c}(p)\leftarrow \mathfrak{c}(p;0,a))=\mathfrak{c}(p-a)$
		\end{itemize}
	\end{lemm}

	%
	
	\subsection{2-polygraphs} Polygraphs are a notion for presentations of higher dimensional categories by generators, relations, relations between relations and so on. Here we recall the notions relevant to this article, namely the notion for presentations of monoids. The polygraphs have one $0-$cell. For a detailed account, see \cite{guiraud2018polygraphs}.
	
	\smallskip
	\noindent A $2-$polygraph with one $0-$cell is a pair $\Sigma=(\Sigma_1,\Sigma_2)$, where $\Sigma_1$ is a set of generators, called $1-$cells, and $\Sigma_2$ is a set of rewriting rules, called generating $2-$cells, $\alpha:u\Rightarrow v$ with $u,v\in\Sigma_1^*$, the free monoid on $\Sigma_1$. We call $u$ and $v$ respectively the source and target of $\alpha$, and denote them by $s_1(\alpha)$ and $t_1(\alpha)$. A $2-$category, respectively a $(2,1)-$category, is a category enriched in categories, respectively in groupoids. We denote by $\Sigma_2^*$ and $\Sigma_2^\top$ the $2-$category, respectively the $(2,1)-$category, generated by the $2-$polygraph $\Sigma$.  A \textit{sphere} in $\Sigma_2^\top$ is a pair $(f,g)$ of $2-$cells in $\Sigma_2^\top$ such that $s_1(f)=s_1(g)$ and $t_1(f)=t_1(g)$.
	
	\smallskip
	\noindent The monoid presented by the $2-$polygraph $\Sigma$ is denoted by $\overline{\Sigma}$, and is defined as
	$$
	\overline{\Sigma}:=(\Sigma_1^*)/\equiv_{\Sigma_2},
	$$
	where $\equiv_{\Sigma_2}$ is the congruence on $\Sigma_1^*$ generated by the $2-$cells $u\Rightarrow v\in \Sigma_2$.
	
	\subsection{Rewriting properties of $2-$polygraphs}
	Call a $2-$polygraph $\Sigma$ \textit{finite} if  $\Sigma_i$ are finite for $i=1,2$. A \textit{rewriting step} in $\Sigma$ is a $2-$cell of the form $wuw'\stackrel{\alpha}{\Rightarrow} wvw'$ where $\alpha:u\Rightarrow v$ is a $2-$cell in $\Sigma_2$, and $w,w'\in\Sigma_1^*$. A \textit{rewriting sequence} in $\Sigma$ is a sequence of rewriting steps. Note that this can be a finite or an infinite sequence. If such a sequence exists from $u$ to $v$, we say that $u$ rewrites into $v$. Call a word $w\in\Sigma_1^*$ \textit{of normal form} if there exists no rewriting step with $w$ as its source. Given a word $w\in\Sigma_1^*$, call $v\in\Sigma_1^*$ a normal form of $w$ if $v$ itself is of normal form, and $w$ rewrites into $v$. We say that the polygraph $\Sigma$ is \textit{terminating} if there are no infinite rewriting sequences in $\Sigma$. We say that $\Sigma$ is confluent if for any words $u,u_1,u_2\in\Sigma_1^*$ such that $u$ rewrites into $u_1$ and $u_2$, there exists a word $v\in\Sigma_1^*$ such that $u_1$ and $u_2$ rewrite into $v$. We say that $\Sigma$ is \textit{convergent} if it is terminating and confluent.
	
	\smallskip
	\noindent A \textit{branching} of $\Sigma$ is a pair $(f,g)$ of $2-$cells in $\Sigma_2^*$ with a common source. We can define an order $\preceq$ on the family of branchings to be generated by the relations
	$$
	(f,g)\preceq (ufv,ugv)
	$$
	for $u,v\in\Sigma_1^*$. A branching that is minimal with respect to this order, is called a \textit{critical branching}.
	
	\smallskip
	\noindent Given a critical branching $(f,g)$ with source $w$ in a convergent $2-$polygraph, denote by $F$ and $G$ respectively the $2-$cells in $\Sigma_2^*$ which starting form $f$, respectively $g$, reduce $w$ to a normal form. We call the $3-$cell $F \Rrightarrow G$ the \textit{confluence diagram} of the critical branching $(f,g)$. More generally, $2-$cells $F,G$ in $\Sigma_2^*$ are called \textit{parallel} if they have the same source and target. Such a pair $(F,G)$ is then called a 2-\textit{sphere}.
	
	\subsection{Coherent presentations of monoids}
	
	\noindent A $(3,1)-$polygraph is a pair $(\Sigma,\Sigma_3)$ with $\Sigma$ a $2-$polygraph, and $\Sigma_3$ a set of generating $3-$cells $A:f\Rrightarrow g$, with $f$ and $g$ parallel $2-$cells in the $(2,1)-$category $\Sigma^\top$, such that $s_1(f)=s_1(g)$ and $t_1(f)=t_1(g)$. Denote by $\Sigma_3^\top$ the free $(3,1)-$category generated by $(\Sigma,\Sigma_3)$.
	
	\noindent A \textit{coherent presentation} of a monoid $\mathcal{M}$ is a $(3,1)-$polygraph $(\Sigma,\Sigma_3)$ such that $\Sigma$ is a presentation of $\mathcal{M}$, and for any $2-$sphere $\gamma$ of $ \Sigma^\top$ there exists a $3-$cell in $\Sigma_3$ with boundary $\gamma$. In other words, a coherent presentation of a monoid is one that consists of generators, relations, and relations between relations.
	
	\smallskip
	\noindent Squier's theorem from \cite{squier1994finiteness} asserts that a convergent presentation of a monoid extended by the generating $3-$cells defined by the confluence diagrams of critical branchings, is a coherent convergent presentation. 

	\subsection{Column presentation of $Pl(C_n)$}\label{subsec: Nohra}
	In \cite{hage2015finite}, Hage constructs a finite convergent presentation of the plactic monoid $Pl(C_n)$, called the \textit{column presentation}. This is the presentation given by the $2-$polygraph $\textbf{ACol}_\bullet=(\textbf{ACol}_{\bullet 1},\textbf{ACol}_{\bullet 2})$ where $\textbf{ACol}_{\bullet 1}=\adm$, and
	$$
	\textbf{ACol}_{\bullet 2} =\{c_1c_2\stackrel{\alpha_{c_1c_2}}{\Longrightarrow} (c_1\leftarrow c_2) |\ c_1,c_2\in\bull,\ c_2\npreceq c_1\}.
	$$
	The critical branchings of this presentation are of the form
	\begin{equation}\label{eq:critical br}
		{\xymatrix{ & (c_1\leftarrow c_2)c_3 \\ c_1c_2c_3 \ar@{=>}[ru]^{\alpha_{c_1c_2}c_3} \ar@{=>}[rd]_{c_1\alpha_{c_2c_3}}\\ & c_1(c_2\leftarrow c_3) }}
	\end{equation}
	for $c_3\npreceq c_2\npreceq c_1$ in $\acol$. The reduction sequence of $c_1c_2c_3$ starting with the upward (downward) arrow is called the \textit{leftmost (rightmost) reduction strategy}.
	\section{$\mathbb{N}-$decorated plactic monoid of type $C$ and its crystal structure}\label{sec:aux results}
	\noindent In this section we introduce the $\mathbb{N}-$decorated plactic monoid of type $C$, to be denoted by $Pl^\mathbb{N}(C_n)$. The purpose of this new object is to overcome the issue of the column insertion not being a map $\bull^2\longrightarrow\bull^2$. Indeed, in \ref{subsec:Col insertion} we have seed that for $c,d\in\bull$, the insertion $(c\leftarrow d)$ may be $0$, $1$, or $2$ columns. Here we adapt the $2-$polygraph \pohag\ so that the presentation is quadratic. To do this, we modify the column presentation \pohag\ of $Pl(C_n)$ to include the empty column $\epsilon$ as a generator, and we add $2-$cells containing $\epsilon$. This approach resembles that of Lebed in \cite{lebed2016plactic} for the plactic monoid of type $A$, via which she establishes braidings from the column and row insertion and uses it to compute the cohomology of $Pl(A_n)$.
	
	\subsection{\pocol\ and insertion as a map}\label{sub:insertion as a map}
	
	\noindent As previously discussed, we set $\text{ACol}(C_n):=\bull\sqcup \{\epsilon\}$, where $\epsilon$ plays the role of an 'empty column'. We have a natural map $
	p:\acol^*\longrightarrow\bull^*
	$
	given by $p(\epsilon)=\varnothing\in\bull^*$, i.e. $\epsilon$ maps to the identity, and $p(c)=c$ for $c\in\acol$ with $c\neq \epsilon$. This map then extends to
	
	\begin{equation}\label{eq:projection P}
	p(c_1\cdots c_k)=\prod_{\tiny{\begin{array}{c}
			1\leq i \leq k\\c_i\neq \epsilon
			\end{array}}}c_i.
	\end{equation}
	
	\noindent We extend the order $\preceq$ to $\acol$ by setting $c\preceq \epsilon$ for all $c\in\acol$, and for $c_1,c_2\neq \epsilon$, we set $c_2\preceq c_1$ if and only if $p(c_2)\preceq p(c_1)$. We now define the insertion as a map $\text{ins}(\ast\ast)=(\ast \leftarrow \ast):\acol^2\longrightarrow \acol^2$ by setting
	$$
	(c_1\leftarrow c_2) :=\begin{cases}
	d_1d_2 &\text{if } (p(c_1)\leftarrow p(c_2))=p(d_1)p(d_2)\in\bull^2\\
	\epsilon d &\text{if } (p(c_1)\leftarrow p(c_2))=d\in\bull,\\
	\epsilon\epsilon&\text{if } (p(c_1)\leftarrow p(c_2))=\varnothing\in\bull^*.
	\end{cases}
	$$
	
	\subsection{Definition of, and convergent presentation for $Pl^\mathbb{N}(C_n)$} Next we construct an $\mathbb{N}-$decorated plactic monoid of type $C$. Consider the $2-$polygraph \pocol\ $=(\mathbf{ACol}_1,\mathbf{ACol}_2)$ where $\mathbf{ACol}_1=\adm$ and
	$$
	\mathbf{ACol}_2=\left\{c_1c_2\Longrightarrow d_1d_2\ |\ 
	c_1,c_2\in\adm\ \text{such that } (c_1\leftarrow c_2)=d_1d_2\ \text{and } c_2\npreceq c_1 
	\right\}
	$$
	
	\noindent Denote
	$$
	Pl^\mathbb{N}(C_n):=\overline{\Sigma}=\left(\adm^*/\equiv_\mathbf{ACol}\right),
	$$
	where $\equiv_\mathbf{ACol}$ is the congruence on $\adm^*$ generated by $\mathbf{ACol}_2$,
	and call $Pl^\mathbb{N}(C_n)$ the $\mathbb{N}-$\textit{decorated plactic monoid of type }$C$ \textit{and rank }$n$. We will soon exhibit the elements of $Pl^\mathbb{N}(C_n)$, and its product. Before we do this, we first show that \pocol\  is a convergent polygraph. We do this by utilizing the map $p$ as defined in \ref{sub:insertion as a map} in order to compare \pocol\ to \pohag.
	
	\smallskip
	\noindent Denote by $|-|$ and $|-|_\bullet$ respectively the length functions in \pocol\ and \pohag. Here the length functions may take as argument either words, or rewriting sequences $s:w\Longrightarrow w'$ in \pocol\ (respectively \pohag). For $w\in\adm^*$, we have $|w| = |p(w)|_\bullet + \#_\epsilon(w)$, where $\#_\epsilon(w)$ counts the number of $\epsilon$ that appear in $w$. 
	
	\smallskip
	\noindent By definition of the rewriting system $\Longrightarrow$, we see that if $w\Longrightarrow w'$ is any $2-$cell in \pocol, then $|w|=|w'|$. To distinguish the generating $2-$cells of \pocol\ that contain $\epsilon$, we denote them by $\stackrel{\epsilon}{\Longrightarrow}$. These are the $2-$cells of the form $c\epsilon \stackrel{\epsilon}{\Longrightarrow}\epsilon c$ for $c\neq\epsilon$.  For a rewriting sequence $s:w\Longrightarrow w'$ in \pocol, denote by $\#_\epsilon(s)$ the number of $\stackrel{\epsilon}{\Longrightarrow}$ that appear in $s$. The following result gives an upper bound of $\#_\epsilon(s)$.
	\begin{prop}\label{prop:some comp}
		Let $w\in\adm^*$ with $|w|=t$, and let
		$$
		s:	w=w_0\Longrightarrow w_1\Longrightarrow \cdots \Longrightarrow w_k
		$$
		be a rewriting sequence in \pocol,
		where $(w_i\Longrightarrow w_{i+1})\in\mathbf{ACol}_2$.
		Then $\#_\epsilon(s)\leq \frac{t^2(3t+1)}{4}$.
	\end{prop}
	\begin{proof}
		For $i=0,1,\dots,k$, denote $t_i=\#_\epsilon(t_i)$. By definition, a generating $2-$cell $u \Longrightarrow u'$ increases the number of $\epsilon$ by at most $1$. Thus $t_i\leq t_{i+1}\leq t_i+1$. Let $0\leq j_1<\dots < j_l\leq k$ be all those indices for which we have $t_{j_{i}+1}=t_{j_i}+1$. We can then split our sequence $s$ into
		$$
		s =  \left(s_1 \Longrightarrow_1 s_2\Longrightarrow_2 \cdots \Longrightarrow_{l-1} s_l\right).
		$$
		where $s_i: w_{j_{i-1}+1}\Longrightarrow \cdots \Longrightarrow w_{j_i}$ for $i=1,\dots,l$, and where $j_0:=-1$. Moreover, as $\Longrightarrow_{i}$ increase the number of $\epsilon$, we see that $\Longrightarrow_i\neq \stackrel{\epsilon}{\Longrightarrow}$, thus we have
		\begin{equation}\label{eq:counting prop}
		\#_\epsilon(s)=\sum_{i=1}^l\#_\epsilon(s_i).
		\end{equation}
		
		\noindent We now compute an upper bound for $\#_\epsilon(s_i)$, for instance for $s_1$. For simplicity, denote $j_1=r$. 
		We have a rewriting sequence
		$
		s_0: w_0\Longrightarrow w_1\Longrightarrow\cdots \Longrightarrow w_{r}
		$
		such that $t_0=t_1=\cdots = t_r=:q$. If $w=c_1\cdots c_t$, denote by $i_1,\dots,i_q$ the indices where $\epsilon$ appears in $w$, so that $c_{i_1}=c_{i_2}=\cdots=c_{i_q}=\epsilon$. Note that for $l=1,\dots,q$ we have that $c_{i_l}$ is preceded by $i_l-1$ columns, of which $l-1$ are equal to $\epsilon$. Hence there are at most $i_l-l$ $2-$cells  $\stackrel{\epsilon}{\Longrightarrow}$ in $s_0$ involving $c_{i_l}=\epsilon$. Thus the number of $2-$cells in $s_0$ that involve $c_{i_l}$ is
		$
		\#_\epsilon(s_0)=\sum_{l=1}^q (i_l-l)$. Note that for each $l$ we have that $c_{i_j}$ for $j=l+1,\dots,q$ are columns to its right. Thus we have that $i_l\leq t-q+l$, which gives us
		$$
		\#_{\epsilon}(s_0)\leq \sum_{l=1}^q(t-q+l)=q(t-q)+\frac{q(q+1)}{2}.
		$$
		Note that we have $q(t-q)=\left(\frac{t}{2}\right)^2-\left(q-\frac{t}{2}\right)^2\leq \frac{t^2}{4}$. As $q\leq t$, we obtain
		$
		\#_\epsilon(s_0)\leq \frac{t^2}{4}+\frac{t^2+t}{2}=\frac{t(3t+1)}{4}.
		$
		We have proven this upper bound for $s_0$, but this proof holds for all $s_i$.
		
		\noindent On the other hand, from $t_0=t_{j_1}\geq 0$, and $t_{j_{i+1}} = t_{j_i}+1$, we obtain $t_{j_l} = t_0+l$. Since $t_{j_l}\leq |w_k|=|w|$, we obtain $l\leq t-t_0\leq t$. Thus we finally have
		$$
		\#_\epsilon(s)=\sum_{i=1}^l\#_\epsilon(s_i)\leq l\frac{t(3t+1)}{4}\leq \frac{t^2(3t+1)}{4}
		$$
		which is what we wanted to show.
	\end{proof}
	\noindent Now using Proposition \ref{prop:some comp}, and the map $p:\acol\longrightarrow \bull$ we are able to study the presentation \pocol\ of $Pl^\mathbb{N}(C_n)$ by comparing it with \pohag. The proof of the following result uses Hage's proof in \cite{hage2015finite} of the convergence of \pohag.
	
	\begin{thm}\label{thm:conv poly}
		The $2-$polygraph \pocol\ is finite and convergent.
	\end{thm}
	\begin{proof}
		The finiteness of \pocol\ is evident, as $\adm$ is clearly finite, and $\mathbf{ACol}_2$ injects into $\adm^2$.
		
		\noindent Note that if a generating $2-$cell $\Longrightarrow$ does not involve $\epsilon$, we have
		\begin{equation}\label{eq:first imp}
		w \Longrightarrow w' \in\textbf{ACol}\quad\text{if and only if}\quad p(w)\Longrightarrow p(w')\in\textbf{ACol}_\bullet.
		\end{equation}
		\noindent On the other hand, as a $2-$cell $\stackrel{\epsilon}{\Longrightarrow}$ does not change the letters $c_i\neq \epsilon$ of a word $w$, or their order in $w$, we have
		\begin{equation}\label{eq:second imp}
		p(w)=p(w') \in \textbf{ACol}_\bullet\quad \text{if and only if}\quad w=w'\ \text{or } w,w'\Longrightarrow^\epsilon w'' \in\mathbf{ACol}.
		\end{equation}
		for some $w''\in \adm$, and $\Longrightarrow^\epsilon$ a rewriting sequence in \pocol\ consisting of $2-$cells of the form $\stackrel{\epsilon}{\Longrightarrow}$. Indeed, $p(w)=p(w')$ means that if $w\neq w'$, then $w$ and $w'$ differ only by the position of $\epsilon$s in them, which can then be reordered via applications of $\stackrel{\epsilon}{\Longrightarrow}$. The converse is evident.
		
		\smallskip
		\noindent Thus any rewriting sequence
		\begin{equation}\label{eq:seq in big}
		s: w\Longrightarrow w_1\Longrightarrow \cdots \Longrightarrow w_k\in\mathbf{ACol}
		\end{equation}
		of length $k$ in \pocol, gives rise to a rewriting sequence
		\begin{equation} \label{eq:seq in small}
		p(s):p(w)\Longrightarrow \cdots \Longrightarrow p(w_k)
		\end{equation}
		in \pohag. By \eqref{eq:second imp} and Proposition \ref{prop:some comp} we see that
		\begin{equation}\label{eq:bound lengths}
		|s|-\frac{|w|^2(|w|+1)}{4}\leq |p(s)|_\bullet\leq |s|.
		\end{equation}
		
		\noindent Thus the existence of an infinite rewriting sequence in \pocol\ implies the existence of an infinite rewriting sequence in \pohag, which is impossible. Hence all rewriting sequences in \pocol\ are finite, so \pocol\ is a terminating polygraph.
		
		\smallskip
		\noindent We now prove confluence of \pocol. Let $w=c_1\cdots c_k\in\adm^*$, and let $w'$ be a normal form of $w$ in \pocol. If $w'=c_1'\cdots c_k'$, since $w'$ is normal, we have $c_{i+1}'\preceq c_i'$. Let $s=\max\{1\leq i\leq k\ |\ c_i'=\epsilon\}$. Note that as $c_s'\preceq c_{s-1}'\preceq \cdots \preceq c_1'$, we have that $c_i'=\epsilon$ for $i\leq s$. Thus we have $w'=\epsilon^sc_{s+1}'\cdots c_k'$. Moreover by \eqref{eq:first imp} we have that $p(w')=c_{s+1}'\cdots c_k'$ is normal in \pohag, so is a normal form of $p(w)$. Then if $w''$ is another normal form of $w$ in \pocol, similarly we have that $w''=\epsilon^rc''_{r+1}\cdots c_k''$ for some $0\leq r\leq k$, and $c_i''\in\adm$. We then have that $p(w'')=c_{r+1}''\cdots c_k''$ is a normal form of $p(w)$ in \pohag. As \pohag is confluent, we have that $p(w')=p(w'')$. Then by \eqref{eq:second imp}, since $w'\Longrightarrow w''$ is impossible due to $w'$ being normal, we have that $w'=w''$. This implies that \pocol\ is confluent. Indeed, if $w\Longrightarrow w_1$ and $w\Longrightarrow w_2$ in \pocol, since this polygraph is terminating, there exist normal forms $w'$ of $w_1$ and $w''$ of $w_2$. As we have rewriting sequences from $w$ to $w'$ and $w''$, we have that $w'$ and $w''$ are normal forms of $w$. But then we obtain $w'=w''$, hence indeed \pocol\ is confluent.
		
		\noindent Thus \pocol\ is a convergent presentation.
	\end{proof}

	\noindent We note that the proof of Theorem \ref{thm:conv of dec} describes explicitly the normal forms in \pocol. We note that the reading map $g:Pl(C_n)\longrightarrow C_n^*$ can be adapted to $g:Pl(C_n)\longrightarrow \bull^*\longrightarrow \adm^*$ by setting $g(T)=c_1c_2\cdots c_k$, where $c_i$ are the columns of $T$ numbered from right to left. The we can describe  $Pl^\mathbb{N}(C_n)$ as a monoid as follows.
	
	\begin{cor}\label{cor:expressing the new monoid}
		$Pl^\mathbb{N}(C_n)=\{(T,s)\ |\ T\in Pl(C_n),\ k\in\mathbb{N}\}$, where we identify $w=\epsilon^k g(T)$ with $(T,k)$. Moreover, the product in $Pl^\mathbb{N}(C_n)$ is given by
		$$
		(T_1,k_1)\ast (T_2,k_2)=((T_1\leftarrow T_2),k_1+k_2+\text{col}(T_1,T_2)),
		$$
		where $\text{col}(T_1,T_2)=|T_1|_\bullet+|T_2|_\bullet-|(T_1\leftarrow T_2)|_\bullet$.
	\end{cor}
	\noindent For a word $w\in\adm^*$, denote by $[w]$ its normal form in $Pl^\mathbb{N}(C_n)$.
	
	\smallskip
	\noindent We note here that block columns, as introduced in \ref{subsec:conseccols} are elements of $\textbf{ACol}$. We also set $\cons(0)=\epsilon$.
	\subsection{Crystal structure on $\adm^*$}\label{sub:crystal structure on ADM}
	The notions of Kashiwara operators and of highest weight are defined in the classical setting as in \ref{sub:typeC}, i.e. for words in $C_n^*$. These are partial operators $f_i,e_i:C_n^*\longrightarrow C_n^*$ that are stable under the plactic relations. Here we adapt these notions to $\adm^*$ by following \ref{sub:crystalbases and monoids}. We first define a reading map $r:\adm^*\longrightarrow C_n^*$ by setting $r(\epsilon)=\varnothing$, the empty word in $C_n^*$, and $r(c)=x_1\cdots x_k$, if $c=x_1\cdots x_k\in\adm$, and then extending $r$ to all of $\adm^*$ to be a morphism of free monoids.
	
	\noindent Let $I=\{1,2,\dots,n\}$, and consider the $I-$labelled graph $B_{\text{adm}}$ with
	\begin{itemize}
		\item vertex set $V(B_{\text{adm}})=\adm$
		\item $E(B_{\text{adm}})$ consists of edges $c \stackrel{i}{\longrightarrow} d$ where $c,d\neq \epsilon$, and $r(d)=f_i.r(c)$.
	\end{itemize}
	
	\begin{prop}
		$B_\text{adm}$ is a crystal basis.
	\end{prop}
	\begin{proof}
		For $i=1,\dots,n$ and $c\in\adm$ such $f_i.r(c)$ is defined, we know that $f_i.r(c)=r(d)$ for some $d\in\adm$. Moreover, such a $d$ is unique. Indeed, if $c=x_1\cdots x_k\in\adm$ is such that $f_i.r(c)$ is defined, we then have $f_i.r(c)=x_1'\cdots x_k'$ with $x_i=x_i'$ for all but one index $1\leq i_0\leq k$. Thus $d$ is entirely determined, hence is unique. Similarly we can show that given $d\in\adm$ such that $e_i.r(d)$ is defined, then $r(c)=e_i.r(d)$ for a unique $c\in\adm$, and thus we have a unique edge $c\stackrel{i}{\longrightarrow}d$. These arguments show that given $c\in\adm$ and $i=1,\dots,n$ we have
		$$
		\begin{array}{c}\#\{e=(c,d)\in E\ |\ l(e)=i\}\leq 1\\ \#\{e=(d,c)\in E\ |\ l(e)=i\}\leq 1
		\end{array}.
		$$
		Moreover, any path $p:c\stackrel{i}{\longrightarrow}c_1\stackrel{i}{\longrightarrow}\cdots\stackrel{i}{\longrightarrow}c_k$ in $B_\text{adm}$ gives rise to a path $r(p):r(c)\stackrel{i}{\longrightarrow}r(c_1)\stackrel{i}{\longrightarrow}\cdots\stackrel{i}{\longrightarrow}r(c_k)$ in $\Gamma_{C_n}$. As $r(p)$ is finite, we have that $p$ must also be finite. Hence $B_\text{adm}$ is indeed a crystal basis.
	\end{proof}

	\noindent Denote by $\Gamma_{\text{adm}}$ the crystal graph arising from the crystal basis $B_{\text{adm}}$. This way we have definitions of Kashiara operators $f_i,e_i$, maps $\varepsilon_i,\varphi_i$, highest weights $w^0$, and connected components $B(w)$ on $\adm^*$. These are notions that can be expressed in terms of the crystal structure on $C_n^*$ as follows
	\begin{prop}\label{prop:perties of new cr}
		Let $w=c_1\cdots c_k\in\adm^*$, $p(w)=c_{i_1}\cdots c_{i_l}$ as in \eqref{eq:projection P}, and $i=1,\dots,n$. Then
		\begin{itemize}
			\item[$i)$] $f_i.w$ is defined if and only if $f_i.r(w)$ is defined, and if $f_i.r(w)=r(c_{i_1})\cdots (f_i.r(c_{i_s}))\cdots r(c_{i_l})$, we have
			$$
			f_i.w=c_1\cdots f_i.c_{i_s}\cdots c_k;
			$$ 
			
			\item[$ii)$]$e_i.w$ is defined if and only if $e_i.r(w)$ is defined, and if $e_i.r(w)=r(c_{i_1})\cdots (e_i.r(c_{i_s}))\cdots r(c_{i_l})$, we have
			$$
			e_i.w=c_1\cdots f_i.c_{i_s}\cdots c_k;
			$$ 
			\item[$iii)$] $\varepsilon_i(w)=\varepsilon_i(r(w))$, and $\varphi_i(w)=\varphi_i(r(w))$, where the left sides of the equations are for $\Gamma_{\text{adm}}$, and the right sides for $\Gamma_{C_n}$;
			\item[$iv)$] $r:B(w)\longrightarrow B(r(w))$ is an isomorphism of labeled graphs
			\item[$v)$] $r(w^0)=r(w)^0$
			\item[$vi)$] $[w]=\epsilon^t w'$, where $r(w')=[r(w)]$ and $t=|w|-|w'|$.
		\end{itemize}
	\end{prop}
	\begin{proof}
		We prove $i), ii), iii)$ simultaneously by induction on $|p(w)|_\bullet$. If $|p(w)|_\bullet = 1$, we have that $w=\epsilon^r c \epsilon^t$ for some $r,t\in\mathbb{N}$ and $c\in\bull$. We then have that $f_i.w$ is defined if and only if $f_i.c$ is defined, if and only if $f_i.r(c)=f_i.r(w)$ is defined in $\Gamma_{C_n}$. Moreover, we have $f_i.w=\epsilon^r f_i.c \epsilon^t$, and $f_i.r(w)=f_i.c$, hence for $|p(w)|_\bullet = 1$ the statement holds. Similarly we prove for $e_i.w$, and it is clear that $\varepsilon_i(w)=\varepsilon_i(r(w))$ and $\varphi_i(w)=\varphi_i(r(w))$ in this case.
		
		\noindent Assume now that $i),ii),iii)$ hold for $|p(w)|_\bullet \leq l$. Let $w=c_1\cdots c_k$ such that $p(w)=c_{i_1}\cdots c_{i_l}c_{i_{l+1}}$. Set $u=c_1\cdots c_{i_l}$, and $v=c_{i_{l+1}}\cdots c_k$. By definition of the actions of $f_i$ on the crystal graph $\Gamma_{\text{adm}}$ we have
		$$
		f_i.(uv)=\left\{\begin{array}{cc}
		(f_i.u)v & \text{if }\varphi_i(u)> \varepsilon_i(v)\\
		u(f_i.v) & \text{if } \varphi_i(u) \leq \varepsilon_i(v)
		\end{array}\right.
		$$
		By induction hypothesis for $ii)$ and $iii)$ we see that $f_i.w$ is defined if and only if $f_i.r(w)$. Assume that this is the case. We differentiate between two possibilities.
		
		\noindent If $\varphi_i(r(u))>\varepsilon_i(r(v))$ we have that $s$ as in the statement of $i)$ is such that $s\leq l$, and since $\varphi_i(u)=\varphi_i(r(u))$, $\varepsilon_i(v)=\varepsilon_i(r(v))$ we have
		$$
		f_i.w=(f_i.u)v=c_1\cdots f_i.c_{i_s}\cdots c_k.
		$$
		If $\varphi_i(r(u))\leq \varepsilon_i(r(v))$, we have that $s$ as in the statement of $i)$ is $s=l+1$, and again we obtain
		$$
		f_i.w=u(f_i.v)=c_1\cdots f_i.c_{i_s}\cdots c_k,
		$$
		which is what we wanted to show.
		
		\smallskip
		\noindent $iv)$ Since $r(f_i.w)=f_i.r(w)$ and $r(e_i.w)=e_i.r(w)$, we have that $r$ indeed maps $B(w)$ to $B(r(w))$. If $w'\in B(w)$, then $|w'|=|w|$, and if $w'=d_1\cdots d_k$ we have $|d_i|=|c_i|$ in $C_n^*$. If we have $r(w)=r(w')$, clearly we must have $d_i=c_i$, hence $w=w'$. Thus $r$ is an injective map. For surjectivity, given $u\in B(r(w))$, we have $u=K.r(w)$, where $K$ is a sequence of Kashiwara operators applied to $r$. But this implies $u=r(K.w)$, hence $r$ is surjective. This way we have proven that $r$ is indeed a labeled graph isomorphism.
		
		\smallskip
		\noindent $v)$ Follows directly from $iv)$.
		
		\smallskip
		\noindent $vi)$ From Corollary \ref{cor:expressing the new monoid}, we have that there exist unique $t\in\mathbb{N}$ and $T\in Pl(C_n)$ such that $w=\epsilon^tg(T)$. Clearly we have $t=|w|-|g(T)|$, as we know that the relations in $\adm^*$ preserve lengths. Note now that
		$$
		[r(w)]=[r(p(w))]=r[p(w)]_\bullet,
		$$
		where $[p(w)]_\bullet$ represents the normal form of $p(w)$ in \pohag. But as \pohag\ presents $Pl(C_n)$ we have that $[p(w)]_\bullet = g(T')$ for some tableau $T'\in Pl(C_n)$, so that we obtain $[r(w)]=T'$. On the other hand we have
		$
		g(T)=p([w])=[p(w)]_\bullet = g(T'),
		$
		which gives us $T=T'$. Hence we finally obtain $r(g(T))=T=T'=r(g(T'))=[r(w)]$, which is what we wanted to show.
	\end{proof}
	
	\noindent Note that $iv)$ from the previous proposition shows that the crystal monoid associated to the crystal graph $\Gamma_{\text{adm}}$ is $C_n^*$. This is due to the fact that the Kashiwara operators, and the map $r$ ignore any $\epsilon$ that may appear as a letter in $w\in\adm^*$. One way to overcome this is to keep track of the length of each word in $\adm^*$ when applying the Kashiwara operators. With this idea in mind, we define the following.
	
	\begin{defin}
		The extended crystal congruence on $\Gamma_{\text{adm}}$ is the congruence $\equiv_\epsilon$ generated by the equivalence relation $\sim_\epsilon$ defined on $\adm^*$ as follows
		$$
		w_1\sim_\epsilon w_2 \quad \text{if}\quad \left\{\begin{array}{l}
		|w_1| = |w_2|,\\
		\phi: B(w_1)\stackrel{\cong}{\longrightarrow}B(w_2)\quad \text{such that\ \ } \phi(w_1)=w_2.
		\end{array}\right.
		$$
		The monoid $\adm^*/\equiv_\epsilon$ is called the extended crystal monoid of type $C$.
	\end{defin}
	\begin{thm}\label{prop:crystal in plactic}
		Let $w_1,w_2\in\adm^*$. We have $w_1\equiv_\epsilon w_2$ if and only if $[w_1]=[w_2]\in Pl^\mathbb{N}(C_n)$. In other words $(\adm^*/\equiv_\epsilon)=Pl^\mathbb{N}(C_n)$.
	\end{thm}
	\begin{proof}
		Let $w_1,w_2\in\adm^*$ such that $w_1\equiv_\epsilon w_2$. We then have $|w_1|=|w_2|$ and a labeled graph isomorphism
		$$
		\gamma: B(r(w_1))\longrightarrow B(w_1)\longrightarrow B(w_2)\longrightarrow B(r(w_2))
		$$
		such that $\gamma(r(w_1))=\gamma(r(w_2))$. This shows that $[r(w_1)]=[r(w_2)]$ in $C_n^*$. Let now $[w_1]=\epsilon^{t_1}g(T_1)$ and $[w_2]=\epsilon^{t_2}g(T_2)$. By Proposition \ref{prop:perties of new cr} we have $T_1=[r(w_1)]=[r(w_2)]=T_2$. Since $|g(T_1)|=g(T_2)$ we obtain $t_1=t_2$, hence indeed $[w_1]=[w_2]$ in $Pl^\mathbb{N}(C_n)$.
		
		\noindent Conversely, suppose that $[w_1]=[w_2]=\epsilon^tg(T)$. We then have that $T=[r(w_1)]=[r(w_2)]$, hence we have a labeled graph isomorphism
		$$
		\beta: B(w_1)\longrightarrow B(r(w_1))\longrightarrow B(r(w_2))\longrightarrow B(w_2)
		$$
		with $\beta(w_1)=w_2$. Moreover, since $|w_1|=|[w_1]|=|[w_2]|=|w_2|$, we see that $w_1\equiv_\epsilon w_2$, which is what we wanted to show.
	\end{proof}
	\noindent We recall here the following result in \cite{kashiwara1991crystal}.
	
	\begin{thm}(\cite{kashiwara1991crystal})
		Let $T\in Pl(C_n)$ be a tableau of highest weight. Then all the elements in the $i-$th row of $T$ are equal to $i$.
	\end{thm}
	\noindent Interpreted in our setting, this result takes the following shape
	
	\begin{cor}\label{cor:hw decorated}
		Let $w\in Pl^\mathbb{N}(C_n)$ be of highest weight. Then
		$$
		w=\prod_{i=1}^k\cons(a_i)
		$$
		for some $k\in\mathbb{N}$, and $a_1\leq a_2\leq \cdots \leq a_k$.
	\end{cor}

	\noindent In sections \ref{sec:Ctrees} and \ref{sec:calculations with c trees} we will construct a graph model to parameterize all the highest weight elements of $\adm^*$.

	\subsection{Reduction sequences in \pocol}
	\label{sec:red sequences}
	\noindent Here we show that reduction sequences in \pocol, in a certain sense, do not depend on the weight of the starting word $w$.
	
	\begin{defin}
		Let $w=c_1\cdots c_k\in\adm^*$. A reduction strategy for $w$ is a sequence $s=(s_1,s_2,\dots)$ of natural numbers with $|s_i|<k$, such that there exists a rewriting sequence
		$$
		w=w_0\Longrightarrow w_1\Longrightarrow w_2\Longrightarrow \cdots \Longrightarrow w_l
		$$
		in \pocol, where if $w_i=c_1^{(i)}\cdots c_k^{(i)}$, then we have
		$$
		w_{i+1}=c_1^{(i)}\cdots (c_{s_i}^{(i)}\leftarrow c_{s_{i+1}}^{(i)})\cdots c_k^{(i)}.
		$$
	\end{defin} 
	\noindent In other words, a reduction strategy for a word $w$ is the data that successively describes the locations in $\{1,\cdots,k\}$ where we apply a $2-$cell.
	
	\smallskip
	\noindent For a word $w\in\adm$, denote by $\text{red}(w)$ the set of all reduction strategies for $w$. Our goal is to compare $\text{red}(w)$ with $\text{red}(f_i.w)$ for $i=1,\dots,n$ and $w\in\adm^*$ such that $f_i.w$ is defined. We note first that by Theorem \ref{prop:crystal in plactic} we obtain the following
	
	\begin{cor}Let $w\in\adm^*$, and $i=1,\dots,n$. Then \label{cor:kash+normal}
		
		\begin{itemize}
			\item[1.] $f_i.w$ is defined if and only if $f_i.[w]$ is defined,
			\item[2.] if $f_i.w$ is defined, then $[f_i.w]=f_i.[w]$.
		\end{itemize}
	\end{cor}
	
	\begin{prop}\label{prop:red lengths}
		Let $w\in\adm^*$ and $i=1,\dots,n$ such that $f_i.w$ is defined. Then $\text{red}(w)=\text{red}(f_i.w).$
	\end{prop}
	\begin{proof}
		We show first that $\text{red}(w)\subset \text{red}(f_i.w)$.
		Let $w=c_1\cdots c_k$ and $i=1,\dots,n$ such that $f_i.w$ is defined, and let $s\in\text{red}(w)$. Let $j$ be the first element of $s$. This means that $c_{j+1}\npreceq c_j$. Consider now $f_i.w=d_1\cdots d_k$. If $f_i$ does not act on $w$ by acting on $c_jc_{j+1}$, then we have $d_j=c_j$ and $d_{j+1}=c_{j+1}$, hence $d_{j+1}\npreceq d_j$. If, conversely, $f_i$ acts on $w$ by acting on $c_jc_{j+1}$, we have $d_jd_{j+1}=f_i.(c_jc_{j+1})$. If we had $d_{j+1}\npreceq d_j$, we would have $d_jd_{j+1}=[d_jd_{j+1}]$, which implies $c_jc_{j+1}=e_i.(d_jd_{j+1})=e_i.[d_jd_{j+1}]=[e_i.(d_jd_{j+1})]=[c_jc_{j+1}]$, i.e. $c_{j+1}\preceq c_j$. However this is impossible due to our assumption. Hence there exists a sequence $s\in\text{red}(f_i.w)$ with first element equal to $j$.
		
		\smallskip
		\noindent As the rewriting system \pocol\ is terminating, we have that $s\in\text{red}(w)$ is finite, and applying the same reasoning for the remaining elements of $s$, we obtain that $s\in\text{red}(f_i.w)$.
		
		\smallskip
		\noindent Similarly we can show that $\text{red}(f_i.w)\subset \text{red}(e_i.(f_i.w))$ and putting these together, we obtain that $\text{red}(w)=\text{red}(f_i.w)$.
	\end{proof}
	
	\begin{cor}
		Let $w\in\adm^*$. Then $\text{red}(w)=\text{red}(w^0)$.
	\end{cor}
	
	\subsection{Confluence diagrams of critical branchings}
	Let $t,u,v\in\adm$. For $w=tuv$ we have at most two reduction sequences in $\textbf{ACol}$, i.e.
	$$
	\text{red}(w)=\{a(w)=(1,2,1,2,\dots), b(w)=(2,1,2,1,\dots)\},
	$$
	where $b(w)$ signifies the rightmost reduction strategy, and $a(w)$ signifies the leftmost reduction strategy. Critical branchings in $\mathbf{ACol}$ are of the form
	\begin{equation}\label{eq:critical br2}
	{\xymatrix{ & (t\leftarrow u)v \\ tuv \ar@{=>}[ru]^{\alpha_{tu}v} \ar@{=>}[rd]_{t\alpha_{uv}}\\ & t(u\leftarrow v) }}
	\end{equation}
	
	\noindent with $v\npreceq u\npreceq t$. One has a critical branching with source $w=tuv$ if and only if $a(w)$ and $b(w)$ are non-empty.
	
	\smallskip
	\noindent In what follows, we are interested in characterizing the lengths of $a(w)$ and $b(w)$. By the termination of $\mathbf{ACol}$ we have that $a(w)$ and $b(w)$ are finite, and if we denote their lengths by $|a(w)|$ and $|b(w)|$, we see that they describe the lengths of the leftmost and rightmost reduction strategies. As $\mathbf{ACol}$ is confluent, we have that the normal forms produced by $a(w)$ and $b(w)$ are equal. Hence the lengths of $a(w)$ and $b(w)$ describe the lengths of these two reduction sequences in the confluence diagram of $w$. Hence we make the following definition.
	\begin{defin}
		Let $t,u,v\in\adm$, and $w=tuv$. The confluence pair of $w$ is the pair $(|a(w)|,|b(w)|)$, denoted by $\text{conf}(w)$.
	\end{defin}
	
	\noindent We can summarize the preceding discussion into the following.
	
	\begin{thm}\label{thm:confdiags}
		Let $t,u,v\in\adm$ and $w=tuv$. Then $\text{conf}(w)=\text{conf}(w^0)$.
	\end{thm}
	
	\noindent Thus in order to understand $\text{conf}(w^0)$, it remains to first characterize words of highest weight of length $3$ in $\adm^*$, and then compute $a(w)$ and $b(w)$. We do this in the remainder of the article.
	
	\section{Highest weight words in $\acol^*$}\label{sec:Ctrees}
	\noindent In this section we describe a graphical model, namely the $C-$trees $(\Gamma_C,s)$, to characterize the highest weight words in $\acol^*$. Here $\Gamma_C$ will be a certain tree, and $s$ will be a labeling of its vertices.
	\subsection{Definition of the $C-$tree}
	\begin{defin}\label{def:Ctree} The $C-$tree is the directed graph $\Gamma_C$ with vertex set
		$
		V(\Gamma_C) =\mathbb{N}_{\geq 1} \times \mathbb{N},	
		$
		and with directed edges of the type
		\begin{equation}\tag{E1}\label{eq:E1}
		(i,0) \longrightarrow (i+1,0),
		\end{equation}
		\begin{equation}\tag{E2}\label{eq:E2}
		(i,j)\longrightarrow (i,j+1).
		\end{equation}
	\end{defin}
	
	\noindent We define two maps on the vertices of $\Gamma_C$ to aid us with the notation of the graph.
	
	\begin{itemize}
		\item projection maps $\text{str},\text{lev}:V(\Gamma_C)\longrightarrow \mathbb{N}$, called the \textit{strand} and \textit{level} of a vertex, by setting
		$$
		\text{str}(i,j)=i;\quad \text{lev}(i,j)=i+\left\lfloor \frac{j+1}{2} \right\rfloor
		$$
		
		\item $\text{type}: E(\Gamma_C)\longrightarrow \{1,2\}$ by setting
		$$
		\text{type}(e)= i,
		$$
		if $e$ is an edge as in (Ei), for $i=1,2$.
	\end{itemize}
	We show that $\Gamma_C$ is indeed a tree. We begin with the following lemma.
	\begin{lemm}\label{lem: 1}
		Let $p:v_1\longrightarrow v_2\longrightarrow \cdots \longrightarrow v_n$
		be a path in $\Gamma_C$, and denote the edges in $p$ by $e_i: v_{i}\longrightarrow v_{i+1}$.  Then:
		
		\begin{itemize}
			\item[$a)$] $\text{str}(v_i)\leq \text{str}(v_{i+1})$ for all $1\leq i \leq n$.
			
			\item[$b)$] If for some $1\leq j \leq n$ we have $\text{type}(e_j)=2$, then $\text{type}(e_k)=2$ for all $k\geq j$.
		\end{itemize}
	\end{lemm}
	\begin{proof} $a)$ is evident by observing in Definition \ref{def:Ctree} that if $(i,j)\longrightarrow (i',j')$ is an edge in $\Gamma_C$, then $i\leq i'$.
		
		\smallskip
		\noindent $b)$ If $\text{type}(e_j)=2$, the target of the edge $e_j$ is a vertex $(i,j)$ with $j\neq 0$. As this vertex is the source of the edge $e_{j+1}$, we have that indeed $\text{type}(e_{j+1})=2$. The result then follows in full by induction.	
		
	\end{proof}
	
	\begin{prop}
		The C-tree is indeed a tree, with root $(1,0)$.
	\end{prop}
	\begin{proof}
		We need to show that given any vertex $v=(i,j)$ of $\Gamma_C$, there exists a unique directed path from $(1,0)$ to $v$.
		
		\smallskip
		\noindent Note that using edges of type 1, for any $i\in\mathbb{N}$ we have a path
		$$
		p_i: (1,0)\longrightarrow (2,0)\longrightarrow \cdots \longrightarrow (i,0).
		$$
		If $p:(1,0)=v_1\stackrel{e_1}{\longrightarrow} \cdots \stackrel{e_{n-1}}{\longrightarrow} v_n=(i,0)$ is any path in $\Gamma_C$, we see that $\text{type}(e_{n-1})=1$, as its target is of the form $(i,0)$. From Lemma \ref{lem: 1} part $b)$ we get that $\text{type}(e_i)=1$ for all $i=1,2,\cdots,n-1$. As for $(1,0)$, there exists a unique edge of type $1$ with $(1,0)$ as source, we get that $e_1:(1,0)\longrightarrow (2,0)$. Inductively we  obtain that
		$
		n=i$ and $e_j:(j,0)\longrightarrow(j+1,0)$
		for $1\leq j \leq i-1$. Thus $p=p_i$, hence $p_i$ is the unique path from $(1,0)$ to $(i,0)$.
		
		\smallskip
		\noindent 	Let now $(i,j)\in V(\Gamma_C)$ with $j\neq 0$. We note that
		$$
		p_{(i,j)}=p_i\ast ((i,0)\longrightarrow (i,1)\longrightarrow \cdots \longrightarrow (i,j))
		$$
		is a path from $(1,0)$ to $(i,j)$ in $\Gamma_C$. To show that this path is unique, pick a path $p:(1,0)=v_1\longrightarrow \cdots \longrightarrow v_n=(i,j)$. Clearly we have $\text{str}(v_1)=1$ and $\text{str}(v_n)=i$. Let $1\leq l \leq n$ be minimal with the property $\text{str}(v_l)=i$. We have that $e_l: v_{l-1}\longrightarrow v_l$, with $\text{str}(v_{l-1})\leq i-1$. Evidently then we have $\text{type}(e_l)=1$, and we get that $v_{l}=(i,0)$. This means that the subpath $(1,0)\longrightarrow \cdots \longrightarrow (i,0)=v_l$ is the unique path $p_i$ from $(1,0)$ to $(i,0)$. In particular we get $l=i-1$. Now as $\text{str}(v_k)=i$ for $k\geq i$, we see that $\text{type}(e_k)=2$ for $k > i$. This forces	$e_{i}:(i,0)\longrightarrow (i,1)$, and inductively we obtain
		$$
		e_k: (i,k-i-1)\longrightarrow	 (i,k-i),\quad \text{for } k>l
		$$
		which gives us $p=p_{(i,j)}$. Thus indeed there exists a unique path from $(1,0)$ to any vertex of $\Gamma_C$, and this concludes the proof of the proposition.
	\end{proof}
	
	\noindent To graphically present the $C-$tree, we make the following notational conventions
	\begin{itemize}
		\item We identify the set of vertices $\Gamma_C$ via the following map\\
		$
		\mathbb{N}_{\geq 1}\times \mathbb{N} \longrightarrow \mathbb{N}_{\geq 1}\cup \{in^\pm \ |\ i\in\mathbb{N}, n\geq 1\}
		$ given by
		$$
		\begin{array}{rcl}
		(i,0) & \longmapsto & i,\\
		(i,2n-1) & \longmapsto & in^+(=in), \\
		(i,2n) & \longmapsto & in^-.
		\end{array}
		$$
	\end{itemize}
	Then the $C-$tree can be graphically presented as follows
	
	
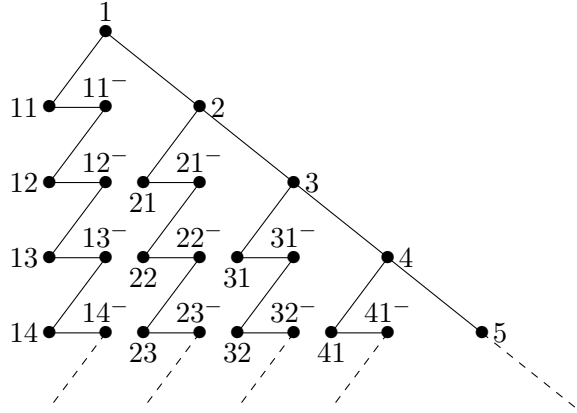
\begin{center}
		\begin{tikzpicture}
		\draw[-] (-0.25,0)--(-1,-1);
		\draw[-] (-1,-1)--(-0.25,-1);
		\draw[-] (-0.25,-1)--(-1,-2);
		\draw[-] (-1,-2)--(-0.25,-2);
		\draw[-] (-0.25,-2)--(-1,-3);
		\draw[-] (-1,-3)--(-0.25,-3);
		\draw[-] (-0.25,-3)--(-1,-4);
		\draw[-] (-1,-4)--(-0.25,-4);
		
		\draw[-] (-0.25,0)--(1,-1);
		
		\draw[-] (1,-1)--(0.25,-2);
		\draw[-] (0.25,-2)--(1,-2);
		\draw[-] (1,-2)--(0.25,-3);
		\draw[-] (0.25,-3)--(1,-3);
		\draw[-] (1,-3)--(0.25,-4);
		\draw[-] (0.25,-4)--(1,-4);
		
		\draw[-] (1,-1)--(2.25,-2);
		\draw[-] (2.25,-2)--(1.5,-3);
		\draw[-] (1.5,-3)--(2.25,-3);
		\draw[-] (2.25,-3)--(1.5,-4);
		\draw[-] (1.5,-4)--(2.25,-4);
		
		\draw[-] (2.25,-2)--(3.5,-3);
		
		\draw[-] (3.5,-3)--(2.75,-4);
		\draw[-] (2.75,-4)--(3.5,-4);
		
		\draw[-] (3.5,-3)--(4.75,-4);
		
		\node at (4.75,-4) {$\cdot$};
		\node at (-0.25,0) {$\cdot$};
		
		\draw[dashed] (-0.25,-4)--(-1,-5);
		\draw[dashed] (1,-4)--(0.25,-5);
		\draw[dashed] (2.25,-4)--(1.5,-5);
		\draw[dashed] (3.5,-4)--(2.75,-5);
		\draw[dashed] (4.75,-4)--(6,-5);
		
		\node at (-0.25,0) {$\bullet$}; \node[above] at (-0.25,0) {\small{$1$}};
		\node at (-1,-1) {$\bullet$}; \node[left] at (-1,-1) {\small{$11$}};
		\node at (-0.25,-1) {$\bullet$}; \node[above] at (-0.25,-1) {\small{$11^-$}};
		\node at (-1,-2) {$\bullet$}; \node[left] at (-1,-2) {\small{$12$}};
		\node at (-0.25,-2) {$\bullet$}; \node[above] at (-0.25,-2) {\small{$12^-$}};
		\node at (-1,-3) {$\bullet$}; \node[left] at (-1,-3) {\small{$13$}};
		\node at (-0.25,-3) {$\bullet$}; \node[above] at (-0.25,-3) {\small{$13^-$}};
		\node at (-1,-4) {$\bullet$};\node[left] at (-1,-4) {\small{$14$}};
		\node at (-0.25,-4) {$\bullet$}; \node[above] at (-0.25,-4) {\small{$14^-$}};

		\node at (1,-1) {$\bullet$}; \node[right] at (1,-1) {\small{$2$}};
		\node at (0.25,-2) {$\bullet$}; \node[below] at (0.25,-2) {\small{$21$}};
		\node at (1,-2) {$\bullet$}; \node[above] at (1,-2) {\small{$21^-$}};
		\node at (0.25,-3) {$\bullet$}; \node[below] at (0.25,-3) {\small{$22$}};
		\node at (0.25,-4) {$\bullet$}; \node[below] at (0.25,-4) {\small{$23$}};
		\node at (1,-4) {$\bullet$}; \node[above] at (1,-4) {\small{$23^-$}};
		\node at (1,-3) {$\bullet$}; \node[above] at (1,-3) {\small{$22^-$}};

		\node at (2.25,-2) {$\bullet$}; \node[right] at (2.25,-2) {\small{$3$}};
		\node at (1.5,-3) {$\bullet$}; \node[below] at (1.5,-3) {\small{$31$}};
		\node at (2.25,-3) {$\bullet$}; \node[above] at (2.25,-3) {\small{$31^-$}};
		\node at (1.5,-4) {$\bullet$}; \node[below] at (1.5,-4) {\small{$32$}};
		\node at (2.25,-4) {$\bullet$}; \node[above] at (2.25,-4) {\small{$32^-$}};
		
		\node at (3.5,-3) {$\bullet$}; \node[right] at (3.5,-3) {\small{$4$}};
		\node at (2.75,-4) {$\bullet$}; \node[below] at (2.75,-4) {\small{$41$}};
		\node at (3.5,-4) {$\bullet$}; \node[above] at (3.5,-4) {\small{$41^-$}};
		\node at (4.75,-4) {$\bullet$}; \node[right] at (4.75,-4) {\small{$5$}};
		
		
		\end{tikzpicture}
	\end{center}
	Via this identification, we allow for the notation
	$$
	\text{lev}(ij)=\text{lev}(ij^-)=i+j.
	$$
	We now define a few relevant subgraphs and a partial order on the $C-$tree.
	
	\begin{itemize}
		\item[1.] We refer to vertices of the form $in$ as \textit{outer vertices}, and to vertices of the form $in^-$ as \textit{inner vertices}.
		\item[2.] Given $i\in\mathbb{N}$, the full subgraph of $\Gamma_C$ on the vertices $\{v\in \Gamma_C\ |\ \text{str}(v)=i\}$ is called the $i-$th \textit{strand} of $\Gamma_C$, and is denoted by $\Gamma^i_C$. Clearly $\Gamma_C^i$ is a tree.
		\item[3.] Given $k\in\mathbb{N}$, the full subgraph of $\Gamma_C$ on the vertices $\{v\in\Gamma_C\ |\ \text{lev}(v)= k\}$ is called the $k-$th \textit{level} of $\Gamma_C$, and is denoted by $\Gamma_{C,k}$. Clearly $\Gamma_{C,k}$ is not connected for $k\neq 0$.
		
		\item[4.] We are also particularly interested in the $k-$th \emph{truncations} of the tree, namely the full subtrees of $\Gamma_C$ on the vertices
		$
		\{ v\in\Gamma_C \ |\ \text{lev}(v)\leq k \}.
		$ We denote it by $\Gamma_C(k)$.
		\item[5.] \noindent It is useful to consider the partial order on $V(\Gamma_C)$ arising out of the directed graph structure of $\Gamma_C$. More precisely, for $v,w\in V(\Gamma_C)$, we say $v\leq w$ if there exists a path $p:v=v_1\rightarrow v_1\rightarrow \cdots \rightarrow v_k=w$ in $\Gamma_C$.
	\end{itemize} 
	
	\subsection{Valuations, and readings of $C-$trees}\label{sub:ReadingsofC}
	
	Here we endow the $C-$tree with a vertex labeling from $\mathbb{N}$.
	\subsubsection{Labelings and valuations of $C-$trees}
	\begin{defin}
		A labeling of rank $k\in\mathbb{N}$ of $\Gamma_C$ is a map
		$
		s:V(\Gamma_C(k))\longrightarrow \mathbb{N}.
		$
		The pair $(\Gamma_C,s)$ is called an $\mathbb{N}-$labeled $C-$tree of rank $k$.
	\end{defin}
	\noindent We illustrate a valuated $C-$tree with its labels marked down at the corresponding vertex, and often times we neglect noting its rank $k$, which is implicitly given by $s$.
	
	\smallskip
	\noindent For a labeled $C-$tree $(\Gamma_C,s)$ of rank $k$ we define a new map $q:V(\Gamma_C(k))\longrightarrow \mathbb{Z}$ on the vertices of $\Gamma_C(k)$, called the \textit{valuation} of $(\Gamma_C,s)$. For $v\in V(\Gamma_C(k))$ with $\text{str}(v)=i$ we set
	$$
	q(v)=\sum_{\tiny{\begin{array}{l}
			ij\in\Gamma_C^i\\
			ij \leq v
			\end{array}}}s(ij)\ -\sum_{\tiny{\begin{array}{l}
			ij^-\in\Gamma_C^i\\
			ij^- \leq v
			\end{array}}}s(ij^-).
	$$
	
	\noindent For a natural number $n$ and a labeling $s$ of rank $k$ of $\Gamma_C$, we say that $s$ is an $n-$\textit{labeling} if $
	0\leq q(v)\leq n$
	for all $v\in V(\Gamma_C(k))$.
	
	\begin{exo} For $k=1,2,3$, we illustrate the labeled $C-$trees of rank $k$ as follows.
		$$
		
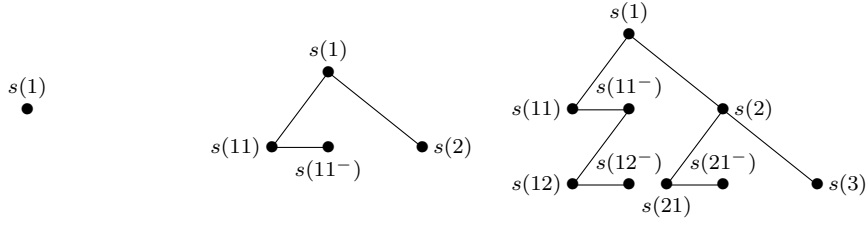
\begin{tikzpicture}
		
		\node at (-4,-1) {\small{$\bullet$}};
		\node[above] at (-4,-1) {\tiny{$s(1)$}};

		\draw[-] (0,-0.5)--(-.75,-1.5);
		\draw[-] (0,-.50)--(1.25,-1.5);
		\draw[-] (-.75,-1.5)--(0,-1.5);
		
		\node at (0,-.5) {\small{$\bullet$}};
		\node[above] at (0,-.5) {\tiny{$s(1)$}};
		
		\node at (1.25,-1.5) {\small{$\bullet$}};
		\node[right] at (1.25,-1.5) {\tiny{$s(2)$}};
		
		\node at (-.75,-1.5) {\small{$\bullet$}};
		\node[left] at (-.75,-1.5) {\tiny{$s(11)$}};
		
		\node at (0,-1.5) {\small{$\bullet$}};
		\node[below] at (0,-1.5) {\tiny{$s(11^-)$}};

		\draw[-] (4,0)--(3.25,-1);
		\draw[-] (3.25,-1)--(4,-1);
		\draw[-] (4,-1)--(3.25,-2);
		\draw[-] (3.25,-2)--(4,-2);
		
		\draw[-] (4,0)--(5.25,-1);
		\draw[-] (5.25,-1)--(6.5,-2);
		
		\draw[-] (5.25,-1)--(4.5,-2);
		\draw[-] (4.5,-2)--(5.25,-2);
		
		\node at (4,0) {\small{$\bullet$}};
		\node[above] at (4,0) {\tiny{$s(1)$}};
		
		\node at (5.25,-1) {\small{$\bullet$}};
		\node[right] at (5.25,-1) {\tiny{$s(2)$}};
		
		\node at (3.25,-1) {\small{$\bullet$}};
		\node[left] at (3.25,-1) {\tiny{$s(11)$}};
		
		\node at (4,-1) {\small{$\bullet$}};
		\node[above] at (4,-1) {\tiny{$s(11^-)$}};
		
		\node at (6.5,-2) {\small{$\bullet$}};
		\node[right] at (6.5,-2) {\tiny{$s(3)$}};
		
		\node at (4.5,-2) {\small{$\bullet$}};
		\node[below] at (4.5,-2) {\tiny{$s(21)$}};
		
		\node at (3.25,-2) {\small{$\bullet$}};
		\node[left] at (3.25,-2) {\tiny{$s(12)$}};
		
		\node at (4,-2) {\small{$\bullet$}};
		\node[above] at (4,-2) {\tiny{$s(12^-)$}};

		\node at (5.25,-2) {\small{$\bullet$}};
		\node[above] at (5.25,-2) {\tiny{$s(21^-)$}};

		\end{tikzpicture}.
		$$
	\end{exo}
	
	\subsubsection{Reading of labeled $C-$trees}
	\noindent Let $(\Gamma_C,s)$ be a $n-$labeled $C-$tree of rank $k$. We define a map $\rho:V(\Gamma_C(k))\longrightarrow C_n^*$ by setting
	$$
	\begin{array}{lcl}
	\rho(i) & = & \mathfrak{c}(s(i)),\\
	\rho(ij) & = & \mathfrak{c}(q(i(j-1)^-),s(ij)),\\
	\rho(ij^{-}) & = & \mathfrak{c}(\overline{q(ij)},s(ij^-)).
	\end{array}
	$$
	
	\smallskip
	\noindent Since $s$ is an $n-$labeling on $\Gamma_C$, we see that $\rho(v)$ is well defined for all vertices $v$ of $\Gamma_C(k)$.
	
	\noindent We define the \textit{reading of the} $t-$\textit{th level of} $(\Gamma_C,s)$ by setting
	
	\begin{equation}\label{eq:reading of cols}
	\omega_t(\Gamma_C,s):= \prod_{l=0}^{t-1} \rho((t-l)l)\prod_{l=1}^{t-1} \rho(l(t-l)^-)
	\end{equation}
	where the product runs over those vertices $v$ with $s(v)\neq 0$. If $s(v)=0$ for all $v$ with $\text{lev}(v)=t$, we set $\omega_t(\Gamma_C,s)=\epsilon$.
	We note right away that $\rho(v)$ are blocks for all $v\in V(\Gamma_C(k))$, and $\omega_t(\Gamma_C,s)$ are products of blocks.

	\begin{defin}\label{def:wordofC}
		Let $(\Gamma_C,s)$ be a $n-$labeled $C-$tree of rank $k$. The reading of, or word of $(\Gamma_C,s)$ is
		$$
		\omega(\Gamma_C,s):=\prod_{t=1}^k \omega_t(\Gamma_C,s)
		$$
		with $\omega_t(\Gamma_C,s)$ as in \eqref{eq:reading of cols}.
	\end{defin}
	
	\noindent Thus we have defined a map
	$
	\omega:\{n-\text{labeled $C$-trees of rank $k$}\}_{k\in\mathbb{N}} \longrightarrow C_n^*.
	$
	While the notation for constructing the word of an $n-$labeled $C-$tree is cumbersome, we can summarize the reading map as follows.
	\begin{itemize}
		\item[1.] Construct a column $\rho(v)$ for each vertex of $V(\Gamma_C(k))$ as follows. If $v=i0$ for some $i$, then $\rho(v)=\cons(s(v))$. Otherwise denote by $v'$ the maximal vertex such that $v'<v$. If $v$ itself is an outer vertex, then $\rho(v)=\cons(q(v'),s(v))$. If $v$ is an inner vertex, then $\rho(v)=\cons(\overline{q(v')},s(v))$.
		\item[2.] For $t\leq k$, the reading of the $t-$th level of $\Gamma_C$ is the product of $\rho(v)$ for $v\in\Gamma_{C,t}$, with the outer vertices read from right to left first, and then the inner vertices read from left to right.
		\item[3.] The word of $(\Gamma_C,s)$ is the product of the readings of its levels.
	\end{itemize}
	
	\begin{exo}\label{ex:a Ctree}
		Let $n=4$ and $T$ the following $\mathbb{N}-$labeled $C-$tree of rank $3$ 
		$$
		
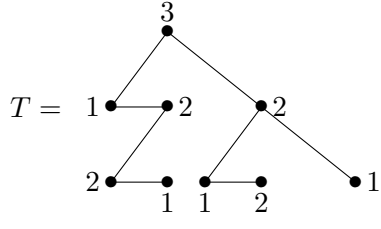
\begin{tikzpicture}
		\draw[-] (-0.25,0)--(-1,-1);
		\draw[-] (-1,-1)--(-0.25,-1);
		\draw[-] (-0.25,-1)--(-1,-2);
		\draw[-] (-1,-2)--(-0.25,-2);
		
		\draw[-] (-0.25,0)--(1,-1);
		\draw[-] (1,-1)--(2.25,-2);
		
		\draw[-] (1,-1)--(0.25,-2);
		\draw[-] (0.25,-2)--(1,-2);
		
		\node at (-.25,0) {\small{$\bullet$}};
		\node[above] at (-.25,0) {\small{$3$}};
		
		\node at (1,-1) {\small{$\bullet$}};
		\node[right] at (1,-1) {\small{$2$}};
		
		\node at (-1,-1) {\small{$\bullet$}};
		\node[left] at (-1,-1) {\small{$1$}};
		
		\node at (-.25,-1) {\small{$\bullet$}};
		\node[right] at (-.25,-1) {\small{$2$}};
		
		\node at (2.25,-2) {\small{$\bullet$}};
		\node[right] at (2.25,-2) {\small{$1$}};
		
		\node at (.25,-2) {\small{$\bullet$}};
		\node[below] at (.25,-2) {\small{$1$}};
		
		\node at (-1,-2) {\small{$\bullet$}};
		\node[left] at (-1,-2) {\small{$2$}};
		
		\node at (-.25,-2) {\small{$\bullet$}};
		\node[below] at (-.25,-2) {\small{$1$}};

		\node at (1,-2) {\small{$\bullet$}};
		\node[below] at (1,-2) {\small{$2$}};
		
		\node at (-2,-1) {\small{$T =$}};
		\end{tikzpicture}.
		$$
		For the $1$st level we have $\rho(1)=\cons(3)=123$ hence $\omega_1=123$. For the $2$nd level we have $\rho(2)=\cons(2)=12$, $\rho(11)=\cons(3,1)=4$, and $\rho(11^-)=\cons(\overline{4},2)=\overline{4}\overline{3}$, hence $\omega_2=124\overline{4}\overline{3}$. For the $3$rd level we have $\rho(3)=\cons(1)=1$, $\rho(21)=\cons(2,1)=3$, $\rho(12)=\cons(2,2)=34$, $\rho(12^-)=\cons(\overline{4},1)=\overline{4}$, and $\rho(21^-)=\cons(\overline{3},2)=\overline{3}\overline{2}$, hence $\omega_3=1334\overline{4}\overline{3}\overline{2}$. Thus we finally obtain
		$$
		\omega(T)=\omega_1\omega_2\omega_3=123124\overline{4}\overline{3}1334\overline{4}\overline{3}\overline{2}.
		$$
	\end{exo}
	
	\noindent Note that in Example \ref{ex:a Ctree} $\omega_1$ is an admissible column, $\omega_2$ is a non-admissible column, and $\omega_3$ is not a column. In what follows, we specify the conditions on the labeling of a $C-$tree $T$ such that $\omega_t(T)$ are all admissible columns. In particular, we illustrate these conditions graphically.
	\subsubsection{Column conditions}
	In this section, we will investigate under what conditions are the readings $\omega_t=\omega_t(\Gamma_C,s)$ of the $t-$th levels admissible columns. Let $(\Gamma_C,s)$ be an $n-$labeled $C-$tree of rank $k$, and $t\leq k$.
	
	\smallskip
	\noindent Denote by $a_l^0$ and $a_l^1$ respectively the leftmost and rightmost elements of the column $\rho((t-l)l)$, and by $b_l^0$ and $b_l^1$ respectively the leftmost and rightmost elements of the column $\rho(l(t-l)^-)$. Then $\omega_t$ is a column if and only if
	\begin{equation}\label{eq:ineq}
	a_l^1< a_{l+1}^0,\quad b_{l}^1< b_{l+1}^0.
	\end{equation}

	\noindent Note that $a_l^0 = q((t-l)(l-1)^-)+1$, $a_l^1 = q((t-l)(l-1)^-)+s((t-l)l)=q((t-l)l)$, $b_l^0 = \overline{q(l(t-l))}$, and $b_l^1 = \overline{q(l(t-l)^-)+1}$.
	Thus the inequalities \eqref{eq:ineq} become
	\begin{equation}\tag{\ref{eq:ineq}$'$}\label{eq:ineq1}
	q((t-l)l) \leq q((t-l-1)(l+1)^-)
	\end{equation}
	and
	\begin{equation}\tag{\ref{eq:ineq}$''$}\label{eq:ineq2}
	q(l(t-l)) \leq q((l+1)(t-l-1)^-).
	\end{equation}
	which we can summarize into
	\begin{equation}\tag{\ref{eq:ineq}$'''$}\label{eq:ineq3}
	q(ij) \leq q((i-1)(j+1)^-)\ \ \text{and}\ \ q((i-1)j^-).
	\end{equation}
	
	\noindent This way we have obtained the following
	\begin{prop}(Column conditions)
		Let $(\Gamma_C,s)$ be an $n-$labeled $C-$tree of rank $k$. The readings $\omega_t=\omega_t(\Gamma_C,s)$ of the $t-$levels are columns if and only if  \eqref{eq:ineq3} are satisfied for all $v=ij\in\Gamma_C(k)$.
	\end{prop}

	\noindent This notation is quite cumbersome. However these inequalities are easy to illustrate on the $C-$tree itself. More specifically, we add a \textit{red edge} $e:u \longrightarrow v$ to the $C-$tree to signify $q(u)\leq q(v)$.
	With these new edges, the labeled $C-$tree is illustrated as follows
	
	
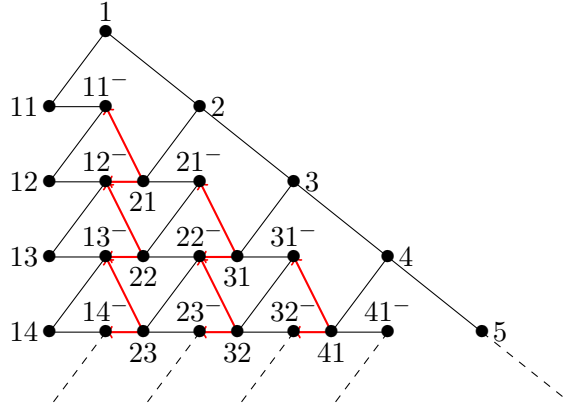
\begin{center}
		\begin{tikzpicture}
		\draw[-] (-0.25,0)--(-1,-1);
		\draw[-] (-1,-1)--(-0.25,-1);
		\draw[-] (-0.25,-1)--(-1,-2);
		\draw[-] (-1,-2)--(-0.25,-2);
		\draw[-] (-0.25,-2)--(-1,-3);
		\draw[-] (-1,-3)--(-0.25,-3);
		\draw[-] (-0.25,-3)--(-1,-4);
		\draw[-] (-1,-4)--(-0.25,-4);
		
		\draw[-] (-0.25,0)--(1,-1);
		
		\draw[-] (1,-1)--(0.25,-2);
		\draw[-] (0.25,-2)--(1,-2);
		\draw[-] (1,-2)--(0.25,-3);
		\draw[-] (0.25,-3)--(1,-3);
		\draw[-] (1,-3)--(0.25,-4);
		\draw[-] (0.25,-4)--(1,-4);
		
		\draw[-] (1,-1)--(2.25,-2);
		\draw[-] (2.25,-2)--(1.5,-3);
		\draw[-] (1.5,-3)--(2.25,-3);
		\draw[-] (2.25,-3)--(1.5,-4);
		\draw[-] (1.5,-4)--(2.25,-4);
		
		\draw[-] (2.25,-2)--(3.5,-3);
		
		\draw[-] (3.5,-3)--(2.75,-4);
		\draw[-] (2.75,-4)--(3.5,-4);
		
		\draw[-] (3.5,-3)--(4.75,-4);
		
		\node at (4.75,-4) {$\cdot$};
		\node at (-0.25,0) {$\cdot$};
		
		\draw[dashed] (-0.25,-4)--(-1,-5);
		\draw[dashed] (1,-4)--(0.25,-5);
		\draw[dashed] (2.25,-4)--(1.5,-5);
		\draw[dashed] (3.5,-4)--(2.75,-5);
		\draw[dashed] (4.75,-4)--(6,-5);
		
		\draw[->,red,thick] (0.25,-2) --(-0.25,-1);
		\draw[->,red,thick] (0.25,-2) --(-0.25,-2);
		\draw[->,red,thick] (0.25,-3) --(-0.25,-2);
		\draw[->,red,thick] (0.25,-3) --(-0.25,-3);
		\draw[->,red,thick] (0.25,-4) --(-0.25,-3);
		\draw[->,red,thick] (0.25,-4) --(-0.25,-4);
		
		\draw[->,red,thick] (1.5,-3) --(1,-2);
		\draw[->,red,thick] (1.5,-3) --(1,-3);
		\draw[->,red,thick] (1.5,-4) --(1,-3);
		\draw[->,red,thick] (1.5,-4) --(1,-4);
		
		\draw[->,red,thick] (2.75,-4) --(2.25,-3);
		\draw[->,red,thick] (2.75,-4) --(2.25,-4);

		\node at (-0.25,0) {$\bullet$}; \node[above] at (-0.25,0) {\small{$1$}};
		\node at (-1,-1) {$\bullet$}; \node[left] at (-1,-1) {\small{$11$}};
		\node at (-0.25,-1) {$\bullet$}; \node[above] at (-0.25,-1) {\small{$11^-$}};
		\node at (-1,-2) {$\bullet$}; \node[left] at (-1,-2) {\small{$12$}};
		\node at (-0.25,-2) {$\bullet$}; \node[above] at (-0.25,-2) {\small{$12^-$}};
		\node at (-1,-3) {$\bullet$}; \node[left] at (-1,-3) {\small{$13$}};
		\node at (-0.25,-3) {$\bullet$}; \node[above] at (-0.25,-3) {\small{$13^-$}};
		\node at (-1,-4) {$\bullet$};\node[left] at (-1,-4) {\small{$14$}};
		\node at (-0.25,-4) {$\bullet$}; \node[above] at (-0.25,-4) {\small{$14^-$}};

		\node at (1,-1) {$\bullet$}; \node[right] at (1,-1) {\small{$2$}};
		\node at (0.25,-2) {$\bullet$}; \node[below] at (0.25,-2) {\small{$21$}};
		\node at (1,-2) {$\bullet$}; \node[above] at (1,-2) {\small{$21^-$}};
		\node at (0.25,-3) {$\bullet$}; \node[below] at (0.25,-3) {\small{$22$}};
		\node at (0.25,-4) {$\bullet$}; \node[below] at (0.25,-4) {\small{$23$}};
		\node at (1,-4) {$\bullet$}; \node[above] at (1,-4) {\small{$23^-$}};
		\node at (1,-3) {$\bullet$}; \node[above] at (1,-3) {\small{$22^-$}};

		\node at (2.25,-2) {$\bullet$}; \node[right] at (2.25,-2) {\small{$3$}};
		\node at (1.5,-3) {$\bullet$}; \node[below] at (1.5,-3) {\small{$31$}};
		\node at (2.25,-3) {$\bullet$}; \node[above] at (2.25,-3) {\small{$31^-$}};
		\node at (1.5,-4) {$\bullet$}; \node[below] at (1.5,-4) {\small{$32$}};
		\node at (2.25,-4) {$\bullet$}; \node[above] at (2.25,-4) {\small{$32^-$}};
		
		\node at (3.5,-3) {$\bullet$}; \node[right] at (3.5,-3) {\small{$4$}};
		\node at (2.75,-4) {$\bullet$}; \node[below] at (2.75,-4) {\small{$41$}};
		\node at (3.5,-4) {$\bullet$}; \node[above] at (3.5,-4) {\small{$41^-$}};
		\node at (4.75,-4) {$\bullet$}; \node[right] at (4.75,-4) {\small{$5$}};
		
		
		\end{tikzpicture}
	\end{center}
	
	\subsubsection{Admissibility conditions}
	
	Assume now that $(\Gamma_C,s)$ is an $n-$labeled $C-$tree satisfying the column conditions. Here we give the conditions for $\omega_t$ to be admissible columns. We make use of Proposition \ref{prop:admissible tree} for this purpose.
	
	\smallskip
	\noindent We recall that
	$$
	\omega_t=\prod_{l=0}^{t-1} \rho((t-l)l)\prod_{l=1}^{t-1} \rho(l(t-l)^-).
	$$
	is a product of blocks. Thus to check admissibility of $\omega_t$, by Proposition \ref{prop:admissible tree}, it suffices to check whether $N_z(\omega_t)\leq z$ for $z$ the rightmost element of a column $\rho((t-l)l)$, or $z$ such that $\overline{z}$ is leftmost element of a column $\rho(l(t-l)^-)$. We note that the last element of $\rho((t-l)l)$ is $q((t-l)l)$, and the first element of $\rho((t-l)l^-)$ is $\overline{q((t-l)l)}$.
	Thus for $\omega_t$ to be admissible, it suffices that the following hold
	$$
	N_{z_l}(\omega_t)\leq z_l\quad \text{for } z_l=q((t-l)l).
	$$
	for $l=0,\dots,k-1$.
	\noindent We note that given $l$, we have
	$$
	\text{Set}_{z_l}(\omega_t)=\{x\in \rho((t-i)i)\ |\ i\leq l\}\sqcup\{x\in \rho((t-i)i^-)\ |\ i\leq l\},
	$$
	thus we have
	$$
	N_{z_l}(\omega_t)=\sum_{i=0}^l |\rho(t-i)i|+\sum_{i=0}^l |(\rho((t-i)i^-)|
	$$
	where we set $s(t0^-)=0$, thus $\rho(t0^-)=\varnothing$ the unit in $C_n^*$.
	By definition of $R$, we have
	$$
	N_{z_l}(\omega_t)=\sum_{i=0}^l \left(s((t-i)i)+ s((t-i)i^-)\right).
	$$
	This way we obtain the following.
	\begin{prop}(Admissibility conditions)
		For an $n-$labeled $C-$tree that satisfies the column conditions, $\omega_t$ is admissible if and only if for all $l\leq t-1$ we have
		\begin{equation}\label{ineq:admissibility}
		\sum_{i=0}^l \left(s((t-i)i)+ s((t-i)i^-)\right)\leq q((t-l)l).
		\end{equation}
	\end{prop}
	\noindent Note that as $q((t-l)l)=q((t-l)(l-1)^-)+s((t-l)l)$, we can write \eqref{ineq:admissibility} as follows
	\begin{equation}\tag{\ref{ineq:admissibility}$'$}
		\sum_{i=0}^{l-1} \left(s((t-i)i)+ s((t-i)i^-)\right)+s((t-l)l^-)\leq q((t-l)(l-1)^-).
	\end{equation}

	\noindent Again the condition in this form is quite cumbersome. Once again we make use of the graphical presentation of the $C-$tree to illustrate this condition.
	
	\noindent Given $t$, for $l\leq t$, the right side of the inequality \eqref{ineq:admissibility} is the map $q$ evaluated at the vertex $(t-l)l$. The left side of the inequality is the sum of all the labels $s(v)$ on the $t-$th level to the right of the vertex $(t-l)l$.
	
	\smallskip
	\noindent We add edges to $\Gamma_C$ as follows
	
	\begin{itemize}
		\item[1.] Label the inner vertices of level $t-1$ with $q(v)$.
		\item[2.] Label the inner vertices of level $t$, with $p(v)=\sum s(v)$, where the sum is taken over all the vertices to the right (and including) $v$.
		\item[3.] Draw a blue edge $e: v\longrightarrow v'$ to signify $p(v)\leq s(v')$.
	\end{itemize}
	For $t=5$, the $C-$tree with these newly appended blue edges to signify the admissibility of $\omega_5(\Gamma_C,s)$ look as follows.
	
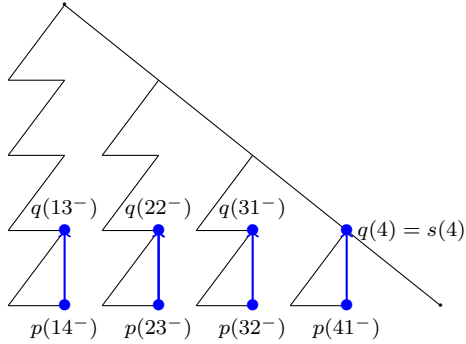
\begin{center}
		\begin{tikzpicture}
		\draw[-] (-0.25,0)--(-1,-1);
		\draw[-] (-1,-1)--(-0.25,-1);
		\draw[-] (-0.25,-1)--(-1,-2);
		\draw[-] (-1,-2)--(-0.25,-2);
		\draw[-] (-0.25,-2)--(-1,-3);
		\draw[-] (-1,-3)--(-0.25,-3);
		\draw[-] (-0.25,-3)--(-1,-4);
		\draw[-] (-1,-4)--(-0.25,-4);
		
		\draw[-] (-0.25,0)--(1,-1);
		
		\draw[-] (1,-1)--(0.25,-2);
		\draw[-] (0.25,-2)--(1,-2);
		\draw[-] (1,-2)--(0.25,-3);
		\draw[-] (0.25,-3)--(1,-3);
		\draw[-] (1,-3)--(0.25,-4);
		\draw[-] (0.25,-4)--(1,-4);
		
		\draw[-] (1,-1)--(2.25,-2);
		\draw[-] (2.25,-2)--(1.5,-3);
		\draw[-] (1.5,-3)--(2.25,-3);
		\draw[-] (2.25,-3)--(1.5,-4);
		\draw[-] (1.5,-4)--(2.25,-4);
		
		\draw[-] (2.25,-2)--(3.5,-3);
		
		\draw[-] (3.5,-3)--(2.75,-4);
		\draw[-] (2.75,-4)--(3.5,-4);
		
		\draw[-] (3.5,-3)--(4.75,-4);
		
		\node at (4.75,-4) {$\cdot$};
		\node at (-0.25,0) {$\cdot$};

		\node at (3.5,-4) {\textcolor{blue}{\small{$\bullet$}}};
		\node[below] at (3.5,-4) {\tiny{$p(41^-)$}};
		\node at (2.25,-4) {\textcolor{blue}{\small{$\bullet$}}};
		\node[below] at (2.25,-4) {\tiny{$p(32^-)$}};
		\node at (1,-4) {\textcolor{blue}{\small{$\bullet$}}};
		\node[below] at (1,-4) {\tiny{$p(23^-)$}};
		\node at (-.25,-4) {\textcolor{blue}{\small{$\bullet$}}};
		\node[below] at (-.25,-4) {\tiny{$p(14^-)$}};

		\node at (3.5,-3) {\textcolor{blue}{\small{$\bullet$}}};
		\node[right] at (3.5,-3) {\tiny{$q(4)=s(4)$}};
		\node at (2.25,-3) {\textcolor{blue}{\small{$\bullet$}}};
		\node[above] at (2.25,-3) {\tiny{$q(31^-)$}};
		\node at (1,-3) {\textcolor{blue}{\small{$\bullet$}}};
		\node[above] at (1,-3) {\tiny{$q(22^-)$}};
		\node at (-.25,-3) {\textcolor{blue}{\small{$\bullet$}}};
		\node[above] at (-.25,-3) {\tiny{$q(13^-)$}};
		
		\draw[->,thick,blue] (2.25,-4)--(2.25,-3);
		\draw[->,thick,blue] (1,-4)--(1,-3);
		\draw[->,thick,blue] (1,-4)--(1,-3);
		\draw[->,thick,blue] (-.25,-4)--(-.25,-3);
		\draw[->,thick,blue] (3.5,-4)--(3.5,-3);
		
		\end{tikzpicture}
	\end{center}
	
	\noindent This way, we have specified conditions for an $n-$labeled $C-$tree $T=(\Gamma_C,s)$ to have its level readings $\omega_t(\Gamma_C,s)$ as admissible columns. If $T$ is of rank $k$, then its reading $\omega(T)$ is defined, and is a product of $k$ admissible columns.
	
	\smallskip
	\noindent We now fix notation. For $k\in\mathbb{N}$, denote by $\mathcal{GT}_k(n)$ the set of $n-$labeled $C-$trees $T$ such that
	\begin{itemize}
		\item[$i)$] $T$ is of rank $k$
		\item[$ii)$] $T$ satisfies the column and admissibility conditions, i.e. $T$ is \textit{admissible}.
	\end{itemize}
	
	\noindent We denote $$
	\mathcal{GT}(n)=\bigsqcup_{k\in\mathbb{N}}\mathcal{GT}_k(n).
	$$
	So $\mathcal{GT}(n)$ consists of all the admissible finite $n-$labeled $C-$trees, and the reading map $w$ can be considered as $\omega:\mathcal{GT}(n)\longrightarrow \adm^*$.
	\begin{remark}
		In what follows, the term $C-$tree will mean an element of $\mathcal{GT}(n)$, unless otherwise specified. We usually denote elements of $\mathcal{GT}(n)$ by $T$, and the labeling $s$ will be implicitly assumed.
	\end{remark}
	\section{Parameterizing words of highest weight in $\adm^*$ via $C-$trees}\label{sec:calculations with c trees}
	
	\subsection{Normal form, and weights of words of $C-$trees}
	\noindent Let $T$ be a $C-$tree of rank $k$. In Definition \ref{def:wordofC} we have defined the reading of $T$ denoted by $\omega(T)$, which is a word in $\adm^*$. In what follows, we will describe the normal form of the word $\omega(T)$ in the $\mathbb{N}-$decorated plactic monoid $Pl^\mathbb{N}(C_n)$. More precisely, for $T\in\mathcal{GT}_k(n)$, we set $q_i(T):=q(i(k-i)^-)$ and we prove the following
	
	\begin{thm}\label{thm:normalofC}
		Let $T\in\mathcal{GT}_k(n)$, and set $q_i=q_i(T)$. Then
		$$
		[\omega(T)]=\prod_{i=0}^{k-1} \cons(q_{k-i}).
		$$
	\end{thm}
	\noindent We will prove this theorem by induction on $k$. For this purpose, we prove several weaker propositions leading to the proof of the theorem. We note right away a crucial consequence of this theorem.
	
	\begin{cor}\label{cor:HW trees}
		Let $T\in\mathcal{GT}_{k}(n)$. Then $\omega(T)$ is of highest weight.
	\end{cor}
	
	\begin{proof}
		By Theorem \ref{thm:normalofC} and Corollary \ref{cor:hw decorated}, we see that $[\omega(T)]$ is highest weight, hence $\omega(T)$ is also of highest weight.
	\end{proof}

	\noindent In other words
	$
	\omega:\mathcal{GT}(n)\longrightarrow \adm^*
	$
	maps every $C-$tree to a highest weight word in $\adm^*$.
	
	\smallskip
	\noindent We now work towards proving Theorem \ref{thm:normalofC}. We begin by establishing some smaller results first.
	
	\noindent Let $k\in\mathbb{N}$ and $a_i\in\mathbb{N}$ with $a_1\geq \cdots \geq a_k \geq a_{k+1}=0$. Consider the labeling $s:V(\Gamma_C(k))\longrightarrow\mathbb{N}$ on $\Gamma_C$  given by
	$$
	s(v)=\begin{cases}
	a_{k-j}-a_{k-j+1} & \text{if } v=ij,\ i+j\leq k,\\
	0 & \text{otherwise}.
	\end{cases}
	$$
	Denote $T(a_i)_{i=1}^k:=(\Gamma_C,s)$.
	
	\begin{prop}\label{prop:standard C trees}
		If $a_i$ as above, with $a_i\leq n$, then $T=T(a_i)_{i=1}^k\in\mathcal{GT}_{k}(n)$. Moreover
		$$
		[\omega(T)]=\omega(T)=\prod_{i=0}^{k-1}\cons(a_{k-i}).
		$$
	\end{prop}
	\begin{proof}
		Note that for $v\in V(\Gamma_C(k))$, if $v=ij^{\pm}$, we have
		\begin{equation}\label{eq:calc of q}
		q(v)=\sum_{w\leq v} s(w)=\sum_{l=0}^js(il)=\sum_{l=0}^j(a_{k-l}-a_{k-l+1})=a_{k-j}\leq n,
		\end{equation}
		Hence $s$ is indeed an $n-$labeling of $\Gamma_C$ of rank $k$.  To show that $T\in\mathcal{GT}(n)$, we need to check whether $T$ satisfies the column and admissibility conditions.
		
		\smallskip
		\noindent For the column conditions, we need to check whether $q(ij)\leq q((i-1)j^-)$ and $q(ij)\leq  q((i-1)(j+1)^-)$ hold for all $i,j$.  By \eqref{eq:calc of q}, we see that $q(ij)=a_{k-j}$, hence the inequalities become $a_{k-j}\leq a_{k-j},a_{k-j-1}$, both of which hold due to the assumptions on $a_i$.
		
		\noindent For the admissibility conditions, we need to check whether the inequality
		$$
		\sum_{i=0}^{r-1} \left(s((l-i)i)+ s((r-i)i^-)\right)\leq q((l-r)(r-1)^-).
		$$
		holds for all $r \leq l-1 \leq k-1$. Again, from our computations of $q$ and $s$  we have $q((l-r)(r-1)^-)=a_{k-r+1}$, and $\sum_{i=0}^{r-1} s((l-i)i)=\sum_{i=0}^{r-1}(a_{k-i}-a_{k-i+1})=a_{k-r+1}$. Thus the admissibility inequality becomes $a_{k-r+1}\leq a_{k-r+1}$, which holds for all $l,r$, hence $T$ satisfies the admissibility conditions as well. Thus we indeed have $T\in\mathcal{GT}_{k}(n)$.
		
		\smallskip
		\noindent Note now that the reading of the $l-$th level of $T$ is as follows
		$$
		\omega_l(T)=\prod_{i=0}^{l-1} \rho((l-i)i) = \prod_{i=0}^{l-1}\cons(q((l-i)(i-1)^-),s((l-i)i))=\prod_{i=0}^{l-1}\cons(a_{k-i+1},a_{k-i}-a_{k-i+1})=\cons(a_{k-l}).
		$$
		\noindent The last equality holds because we set $a_{k+1}=0$, and in general we have $\cons(a,b)\cons(a+b,c)=\cons(a,b+c)$.
		
		\smallskip
		\noindent Thus we have that the reading of $T$ is
		$$
		\omega(T)=\prod_{l=0}^{k-1} \omega_l(T)=\prod_{l=0}^{k-1}\cons(a_{k-l+1}),
		$$
		which is what we wanted to show. As we have $a_{k-l}\leq a_{k-l-1}$, we can see that $\omega(T)$ is indeed of normal form, thus we have $[\omega(T)]=\omega(T)$.
	\end{proof}
	\noindent Proposition \ref{prop:standard C trees} establishes the existence of certain $C-$trees, and computes their reading. We call a $C-$tree of the form $T(a_i)_{i=1}^k$ with $a_1\geq \dots \geq a_k$ a \textit{standard C-tree}. The significance of these particular $C-$trees is that $\omega(T)$ is of normal form in in $\adm^*$, and it is of highest weight according to Corollary \ref{cor:hw decorated}.
	
	\smallskip
	\noindent The following proposition asserts that given a $C-$tree $T\in\mathcal{GT}_{k+1}(n)$, and replacing its $k-$th truncation with a certain standard $C-$tree, produces another $C-$tree, and we compute its reading.
	
	\begin{prop}\label{prop:helper2 for normal}
		Let $T=(\Gamma_C,s)\in\mathcal{GT}_{k+1}(n)$, and let $T_0$ be the $k$th truncation of $T$. Let $Q=T(q_i(T_0))_{i=1}^{k}$, and denote its labeling by $s_Q$. Let $T'=(\Gamma_C,s')$ of rank $k+1$ where $s'$ is given by
		$$
		s'(v)=\begin{cases}
		s_Q(v) & \text{if } \text{lev}(v)\leq k,\\
		s(v) & \text{otherwise}.
		\end{cases}
		$$
		Then $T'\in\mathcal{GT}_{k+1}(n)$, $q_i(T)=q_i(T')$, and $\omega(T')=\omega(Q)\omega_{k+1}(T)$.
	\end{prop}
	\begin{proof}
		Again we check that $s'$ is an $n-$valuation, and that $T'$ satisfies the column and admissibility conditions. Note that for $v\in V(\Gamma_C)$ with $\text{lev}(v)\leq k$ we have $s'=s_Q$, and by Proposition \ref{prop:standard C trees} we see that $s'$ satisfies these conditions. It remains to check for vertices $v$ with $\text{lev}(v)=k+1$.
		
		\smallskip
		\noindent Note that $q_i(Q)=q_i(T_0)$, hence we have $q_{T'}(i(k-i)^-)=q_i(Q)=q_i(T_0)=q_T(i(k-i)^-)$. Moreover, we have
		$$
		q_i(T')=q_i(T_0)+s(i(k+1-i))-s(i(k+1-i)^-)=q_i(T).
		$$
		Hence the $n-$labeling, column, and admissibility conditions for vertices $v\in\Gamma_C$ with $\text{lev}(v)=k+1$ in $T'$ are identical with those in $T$, thus $T'$ also satisfies these conditions, and we have $T'\in\mathcal{GT}_{k+1}(n)$. The reading $\omega(T')$ is clearly $\omega(Q)\omega_k(T)$.
	\end{proof}
	\noindent Next we explicitly calculate an insertion of the last two columns of certain $C-$trees.
	
	\begin{prop}\label{prop:third dude}
		Let $T\in\mathcal{GT}_{k+1}(n)$ and let $Q$ be its $k$th truncation. Assume that $Q$ is a standard $C-$tree, i.e. $Q=T(a_i)_{i=1}^{k}$ for some $a_i\in\mathbb{N}$. Let $T'=(\Gamma_C,s')$ of rank $k+1$ with $s'$ defined as follows
		$$
		s'(v)=\begin{cases}
		s((i+1)(k-i)) & \text{if } v=i(k-i)\ i=1,\dots,k\\
		q_i(T)-q((i+1)(k-i)) & \text{if } v=i(k+1-i)\ i=1,\dots,k+1\\
		s(v) & \text{otherwise}.
		\end{cases}
		$$
		Then $T'\in\mathcal{GT}_{k+1}(n)$, $q_i(T')=q_i(T)$, and
		$$
		\omega(T')=\prod_{l=0}^{k-1} \omega_l(T) [\omega_k(T)\omega_{k+1}(T)].
		$$
	\end{prop}
	\begin{proof}
		We note that $\omega_k(T)=\omega_k(Q)=\cons(a_{k-k+1})=\cons(a_1)$, and moreover $a_1=q_1(Q)=q_T(0k)$. By definition we have
		$$
		\omega_{k+1}(T)=\prod_{i=0}^{k}\rho((k+1-i)i)\prod_{i=1}^k \rho(i(k+1-i)^-).
		$$
		Note now that
		$
		\rho(k+1)=\cons(s(k+1)),\ \rho(k1)=\cons(s(k),s(k1))
		$, and for $i>1$ we have
		$$
		\rho((k+1-i)i)=\cons(q((k+1-i)(i-1)^-),s((k+1-i)i))=\cons(a_{k+1-i},s((k+1-i)i)),
		$$
		and
		$$
		\rho(i(k+1-i)^-)=\cons(\overline{q(i(k+1-i))},s(i(k+1-i))).
		$$
		We now use Lemma \ref{lemm: ref for n 2} to compute the insertion $\omega_k\leftarrow \omega_{k+1}$. Note first that
		$$
		\omega_k'=\left(\omega_k\longleftarrow \prod_{i=0}^{k-1} \rho((k+1-i)i)\right) = \prod_{i=0}^{k-1}\rho((k+1-i)i)\omega_k. 
		$$
		Indeed, $\omega_k=\cons(a_1)$, and the for largest element of $\rho((k+1-i)i)$ for $i\leq k$ is
		$$
		a_{k+1-i}+s((k+1-i)i)=q_T((k+1-i)i)\leq q_T((k-i)i^-)\leq a_i\leq a_1.
		$$
		Since $(\omega_k\leftarrow \rho(1k)\rho(1k^-))=(\cons(a_1)\leftarrow \cons(a_1;s(1k),s(1k^-))=\cons(a_1+s(1k)-s(1k^-))=\cons(q_1)$, we have
		$$
		(\omega_k'\leftarrow \rho(1k))\rho(1k^-))=\prod_{i=0}^{k-1}\rho((k+1-i)i)\cons(q_1).
		$$
		Denote $\overline{R_i}=\rho(i(k+1-i)^-)$. Since $q_2+s(2(k-1)^-)=q(2(k-1))\leq q_1$, we have
		$$
		\cons(q_1)=\cons(q_2)\cons(q_2,s(2(k-1)^-))\cons(q(2(k-1)),q_1-q(2(k-1))),
		$$
		and thus
		$$
		\left(\cons(q_1)\leftarrow \overline{R_1} \right)=(\cons(q_1)\leftarrow \cons(\overline{q(2(k-1))},s(2(k-1)^-))=\cons(q_2)\cons(q(2(k-1)),q_1-q(2(k-1))).
		$$
		Similarly we obtain
		$$
		\left(\cons(q_1)\leftarrow \overline{R_1}\cdots \overline{R_k}\right)=\cons(q_{k+1})\prod_{i=0}^{k-1}\cons(q(k+1-i)i),q_{k-i}-q((k+1-i)i).
		$$
		Putting together these calculations, we obtain
		$$
		[\omega_k\omega_{k+1}]=\prod_{i=0}^k\rho((k+1-i)i)\cons(q_k)\prod_{i=1}^k\cons(q(k+1-i)i),q_{k-i}-q((k+1-i)i).
		$$
		We can then see that $T'$ is such that $\omega(T')=\omega(Q)[\omega_k\omega_{k+1}]$, hence $T'\in\mathcal{GT}_{k+1}(n)$.
	\end{proof}
	
	\noindent We note here a special case of the previous result.
	
	\begin{cor}\label{cor:normal form of C}
		Let $T$ and $Q$ be as in Proposition \ref{prop:third dude}. Assume that $\omega_{k+1}(T)$ contains no barred letters. Then
		$$
		\omega(T')=[\omega(T)]=\prod_{i=0}^{k-1}\omega_i(Q)[\omega_k(Q)\omega_{k+1}(T)]=\prod_{i=0}^{k}\cons(q_{k+1-i})
		$$
	\end{cor}
	
	\noindent We are now ready to prove Theorem \ref{thm:normalofC}.
	
	\begin{proof}[Proof of Theorem \ref{thm:normalofC}]
		We prove it by induction on $k$. For $k=1$ we have
		$$
		T = \ \stackrel{p}{\bullet},
		$$
		for some $p\leq n$. The reading of $T$ is $\omega(T)=\cons(p)=12\cdots p$. We note that this is standard, and that $q_0(T)=p$, thus indeed we have
		$$
		[\omega(T)]=\omega(T)=\cons(p)=\cons(q_0).
		$$
		We illustrate the proof for $k=2$ as well. Let $T\in\mathcal{GT}_2(n)$. We have that $T$ is of the form
		$$
		
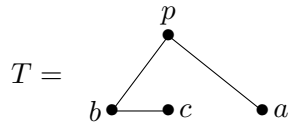
\begin{tikzpicture}
		\draw[-] (-0.25,0)--(-1,-1);
		\draw[-] (-0.25,0)--(1,-1);
		\draw[-] (-1,-1)--(-0.25,-1);
		
		\node at (-.25,0) {\small{$\bullet$}};
		\node[above] at (-.25,0) {\small{$p$}};
		
		\node at (1,-1) {\small{$\bullet$}};
		\node[right] at (1,-1) {\small{$a$}};
		
		\node at (-1,-1) {\small{$\bullet$}};
		\node[left] at (-1,-1) {\small{$b$}};
		
		\node at (-.25,-1) {\small{$\bullet$}};
		\node[right] at (-.25,-1) {\small{$c$}};
		
		\node at (-2,-.5) {\small{$T=$}};
		\end{tikzpicture}
		$$
		for some $p,a,b,c\leq n$ that satisfy the column and admissibility conditions. The reading of $T$ is then
		$$
		\omega(T)=\cons(p)\cons(a)\cons(p;b,c)=c_1c_2.
		$$
		Then the normal form of $\omega(T)$ is $[\omega(T)]=(c_1\leftarrow c_2)$. We note that the admissibility condition asserts that $a+c\leq p$, hence we have $a\leq p$, thus $(\cons(p)\leftarrow \cons(a))=\cons(a)\cons(p)$. Thus we have
		$$
		[\omega(T)]=\cons(a)\left(\cons(p)\leftarrow \cons(p;b,c)\right)=\cons(a)\left(\cons(p+b\leftarrow\cons(\overline{p+b},c)\right)=\cons(a)\cons(p+b-c)=\cons(q_1(T))\cons(q_0(T)).
		$$
		Assume now that the statement of the theorem holds for $k$. Let $T\in\mathcal{GT}_{k+1}(n)$, and denote by $T_0$ its $k-$truncation, and let $Q=T(q_i(T_0))_{i=1}^k$. By induction hypothesis, we have that $[\omega(T)]=\omega(Q)$. Let $T'$ be the tree as in Proposition \ref{prop:helper2 for normal}. Then for the congruence of the $\mathbb{N}-$decorated plactic monoid $\equiv_\epsilon$ on $\adm^*$ we have 
		\begin{equation}\label{eq:111}
		\omega(T)=\omega(T_0)\omega_{k+1}(T)\equiv_\epsilon \omega(Q)\omega_{k+1}(T)=\omega(T').
		\end{equation}
		
		\noindent The $C-$tree $T'$ satisfies the conditions of Proposition \ref{prop:third dude}. Let now $T''$ be the $C-$tree defined as in that proposition. We then have
		
		\begin{equation}\label{eq:222}
		\omega(T')=\prod_{l=1}^{k+1}\omega_l(T)\equiv_\epsilon\prod_{l=1}^{k-1}\omega_l(T')[\omega_k(T)\omega_{k+1}(T)]=\omega(T'').
		\end{equation}
		
		\noindent We note that $T''$ is such that $\omega(T'')$ contains no barred letters. Indeed, in Proposition \ref{prop:third dude} we can see that $s(v)=0$ for any $v=ij^-$. Let $T''_0$ be the $k-$truncation of $T''$, and let $Q''$ be its standard form, and let $T'''$ be as in Proposition \ref{prop:helper2 for normal}. By induction hypothesis we have
		
		\begin{equation}\label{eq:333}
		\omega(T'')=\omega(T''_0)\omega_{k+1}(T'')\equiv_\epsilon \omega(Q'')\omega_{k+1}(T'')=\omega(T''').
		\end{equation}
		
		\noindent As the last level of $T'''$ is the same as that of $T''$, we have that $T'''$ contains no barred letters. Moreover, the $k-$truncation of $T'''$ is standard, hence by Corollary \ref{cor:normal form of C}, we have a $C-$tree $T''''$ such that
		\begin{equation}\label{eq:444}
		\omega(T'''')=[\omega(T''')]=\prod_{i=1}^{k+1}\cons(q_{k+1-i}(T''')).
		\end{equation}
		From \eqref{eq:111},\eqref{eq:222},\eqref{eq:333} we obtain
		$
		\omega(T)\equiv_\epsilon \omega(T')\equiv_\epsilon \omega(T'')\equiv_\epsilon \omega(T''').
		$, and since $T'$, $T''$, and $T'''$ were produced via Propositions \ref{prop:helper2 for normal}, \ref{prop:third dude}, we have that
		$
		q_i(T''')=q_i(T'')=q_i(T')=q_i(T),
		$
		hence we have
		$
		[\omega(T)]=[\omega(T''')]=\prod_{i=1}^{k+1}\cons(q_{k+1-i}),
		$ and we obtain
		$$
		[\omega(T)]=\prod_{i=1}^{k+1}\cons(q_{k+1-i})
		$$
		which is what we wanted to show. The statement of the Theorem then follows by induction on $k$.
	\end{proof}
	
	\noindent Given a $C-$tree $T\in\mathcal{GT}_k(n)$, we call $T(q_i(T))_{i=1}^k$ its \textit{normal form}.
	
	\smallskip
	\noindent Theorem \ref{thm:normalofC} also shows what kind of Kashiwara operators act on $\omega(T)$ for $T\in\mathcal{GT}(n)$.
	
	\begin{cor}\label{cor:action of f}
		Let $T\in\mathcal{GT}(n)$. Then $f_i.\omega(T)$ exists if and only if $i=q_j(T)$ for some $0\leq j \leq \text{rank}(T)$.
	\end{cor}
	\noindent We now note a few more consequences of this result.
	
	\smallskip
	\noindent Given a $C-$tree $T\in\mathcal{GT}(n)$, we see that the normal form of $\omega(T)$ is entirely determined by the values of $q_i(T)$. In particular, we have
	\begin{cor}\label{cor:qi are preserved}
		Let $T_1,T_2\in\mathcal{GT}(n)$. Then
		$$
		[\omega(T_1)]=[\omega(T_2)]\quad \Longleftrightarrow\quad  q_i(T_1)=q_i(T_2)
		$$
		for all $i=0,1,\cdots,k$.
	\end{cor}
	
	\noindent Given a $C-$tree $T\in\mathcal{GT}_k(n)$, we say that $T'\in\mathcal{GT}_l(n)$ is a \textit{sub C-tree} of $T$ if $l\leq k$, and there exists $1\leq r\leq k-l$ such that 
	$$
	s'(ij^\pm)=s((r+i)j^\pm) \quad \text{for all } v=ij^\pm\ \text{with } i+j\leq l.
	$$
	In other words, the $C-$tree $T$ contains a copy of $T'$ with root at some vertex $r0$.
	
	\noindent Let now $T\in\mathcal{GT}(n)$, and $T'$ a sub C-tree of $T$. Let $Q$ be the normal form of $T'$, and let $T_1$ be the $C-$tree obtained when replacing $T'$ by $Q$ in $T$. This will indeed be a $C-$tree, and since $q_i(Q)=q_i(T')$, we have that $q_i(T)=q_i(T_1)$. In particular, we have $\omega(T)\equiv_\epsilon \omega(T_1)$. We note this in the following
	
	\begin{cor}
		Let $T'\subset T$ be $C-$trees, and let $T_1$ be the $C-$tree obtained when replacing $T'$ with its standard C-tree $R$. Then $T_1$ is indeed a $C-$tree, and $[\omega(T)]=[\omega(T_1)]$.
	\end{cor}
	
	\subsection{Constructing $C-$trees from highest weight words}\label{sub:Treesfromwords}
	\noindent In the previous subsection we established where the image of the map $\omega:\mathcal{GT}(n)\longrightarrow \adm^*$ is. In particular, if by $HW^n$ we denote the set of highest weights in $\adm^*$, we have seen that $\omega(T)\in HW^n$. Thus we can consider $w$ as a map $\omega:\mathcal{GT}(n)\longrightarrow HW^n$. Here we will show that this map is surjective. We do this by constructing a map $\mathcal{T}:HW^n\longrightarrow \mathcal{GT}(n)$ such that $\omega(\mathcal{T}(u))=u$ for any $u\in HW^n$. We construct this map inductively on the length of elements of $HW^n$. Denote by $HW_k^n$ the subset of $HW^n$ consisting of words of length $k$.
	
	\smallskip
	\noindent Recall the following result,.
	\begin{lemm}(\cite{lecouvey2002schensted})\label{lem:Lecouvey}
		Let $w_1,w_2\in C_n^*$. The word $w_1w_2$ is of highest weight if and only if
		\begin{itemize}
			\item $w_1$ is a word of highest weight
			\item $\varepsilon_i(w_2)\leq \varphi_i(w_1)$ for all $i=1,2,\dots,n$.
		\end{itemize}
	\end{lemm}
	\begin{remark}
		From our discussion in \ref{sub:crystal structure on ADM}, this result can be adapted to $w_1,w_2\in\adm^*$.
	\end{remark}
	\noindent For $k=1$, by Corollary \ref{cor:hw decorated} we have $c\in HW_1^n$ if and only if $c=\cons(p)$ for some $p\leq n$. So we have
	$
	HW_1^n=\{\cons(p)\ |\ p\leq n\},
	$
	and we define a map $\mathcal{T}_1:HW_1^n\longrightarrow \mathcal{GT}_1(n)$ given by
	$$
	\cons(p) \longmapsto \ \stackrel{p}{\small{\bullet}},
	$$
	and clearly $\omega(\mathcal{T}_1(c))=c$ for all $c\in HW_1^n$.
	
	\smallskip
	\noindent Suppose now that we have established a map $\mathcal{T}_k:HW_k^n\longrightarrow \mathcal{GT}_k(n)$ such that $\omega(\mathcal{T}_k(u))=u$. We now construct a map $\mathcal{T}_{k+1}:HW_{k+1}^n\longrightarrow \mathcal{GT}_{k+1}(n)$ by utilizing this one.
	
	\smallskip
	\noindent Let $u=c_1\cdots c_{k+1}\in HW_{k+1}^n$. Set $t=c_1\cdots c_k$, and by induction hypothesis, let $T=(\Gamma_C,s)\in\mathcal{GT}_k(n)$ be such that $\omega(\mathcal{T}(t))=t$.
	
	\smallskip
	\noindent Let now $l=|g(c_{k+1})|$ in $C_n^*$. If $l=0$, i.e. $c_{k+1}=\epsilon$, then we set $\mathcal{T}_{k+1}(u)$ to be the $C-$tree $T'\in\mathcal{GT}_{k+1}(n)$ whose $k-$truncation is $T$, and its $k+1$th level has all vertex labels equal to $0$. Suppose now that $l=1$, i.e. $c_{k+1}=x$ for some $x\in C_n$. By Lemma \ref{lem:Lecouvey}, we see that
	$
	\varepsilon_i(x)\leq \varphi_i(t)\ \text{for } i=1,\dots,n.
	$
	As $x$ is a one-letter word, we see that $\varepsilon_i(x)\in\{0,1\}$. By Corollary \ref{cor:action of f}, we see that $\varphi_i(v)>0$ only if $i=q_r(T)=q_r$ for some $0\leq r \leq k$. Thus we have $\varepsilon_i(x)>0$ only if $i=q_r$ for some $0\leq r \leq k$. This means that the $x\in C_n$ which satisfy the conditions of Lemma \ref{lem:Lecouvey} are
	$
	x\in \{1,q_r+1,\overline{q_r}\}.
	$ We then define $\mathcal{T}_{k+1}=(\Gamma_C,s')$ by setting $s'(v)=s(v)$ for $\text{lev}(v)\leq k$, and
	\begin{itemize}
		
		\item[1.] if $x=q_r+1$, and $r$ is minimal with this property, then we set $s'(r(k+1-r))=1$,
		\item[2.] if $x=\overline{q_r}$ and $r$ is minimal with this property, then we set $s'(r(k+1-r)^-)=1$,
		\item[3.] if $x=1$ and $q_r\neq 0$, then we set $s'((k+1)0)=1$.
	\end{itemize}
	
	\noindent and $s'(v)=0$ for all other vertices. The fact that $T'\in\mathcal{GT}(n)$ follows from our minimal choice of $r$. Moreover by construction we have $\omega(\mathcal{T}_{k+1})=u$.
	
	\smallskip
	\noindent We now construct the map $\mathcal{T}_{k+1}:HW_{k+1}^n\longrightarrow \mathcal{GT}(n)$ in full by induction on the length $l=|g(c_{k+1})|$ in $C_n^*$. Indeed, say that we have $g(c_{k+1})=dx$, and that $g(w)=g(t)x$.
	Again by Lemma \ref{lem:Lecouvey} we have that $g(t)$ is of highest weight, and $
	\varepsilon_i(x)\leq \varphi_i(g(t))\ \text{for } i=1,\dots,n$. By induction hypothesis on $l=|g(c)|$, since $|g(d)|<|g(c)|$ there exists a $C-$tree $\mathcal{T}_{k+1}(t)\in\mathcal{GT}_{k+1}(n)$ with valuation $s$ such that $\omega(T_{k+1}(t))=t$. Using the same argument as in the case for $|g(c)|=1$, since here $x\neq 1$, we have that 
	$
	x\in\{q_r+1,\overline{q_r}\}
	$
	for some $0\leq r \leq k$.
	Again we apply the same recipe by defining a $C-$tree $\mathcal{T}_{k+1}(w)=(\Gamma_C,s')$ by  setting $s'(v)=s(v)$ for if $s(v)\neq 0$, and
	\begin{itemize}
		\item[1.] if $x=q_r+1$, and r is maximal with this property, then we set $s'(r(k+1-r))=s(r(k+1-r))+1$
		\item[2.] if $x=\overline{q_r}$, and r is minimal with this property, then we set $s'(r(k+1-r)^-)=s(r(k+1-r))+1$
	\end{itemize}
	
	\noindent By construction we have that indeed $\omega(\mathcal{T}_{k+1})(w)=w$.
	
	\smallskip
	\noindent We then define $\mathcal{T}:HW^n\longrightarrow\mathcal{GT}(n)$ by setting $\mathcal{T}(w)=\mathcal{T}_k(w)$ if $|w|=k$ in $\adm^*$. Thus we have proven the first part of the following
	
	\begin{thm}\label{thm:kindamain}
		The map $\mathcal{T}:HW^n\longrightarrow \mathcal{GT}(n)$ is such that
		$$
		\omega(\mathcal{T}(u))=u
		$$
		for all $u\in HW^n$. Moreover, $\mathcal{T}=\omega^{-1}$.
	\end{thm}
	\begin{proof}
		Since given any $u\in HW^n$ we have $\omega(\mathcal{T}(u))=u$, the map $w$ is surjective.
		
		\smallskip
		\noindent To show that $\mathcal{T}=w^{-1}$, it suffices to prove that $w$ is injective, which is proven by induction. We illustrate here the fact that $\omega:\mathcal{GT}_2(n)\longrightarrow \adm^*$ is injective. Indeed, let $T_1,T_2\in\mathcal{GT}_2(n)$ be such that $\omega(T_1)=\omega(T_2)$, i.e.
		
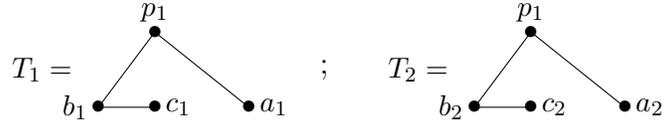
		\begin{center}
		\begin{tikzpicture}
		\draw[-] (-0.25,0)--(-1,-1);
		\draw[-] (-0.25,0)--(1,-1);
		\draw[-] (-1,-1)--(-0.25,-1);
		
		\node at (-.25,0) {\small{$\bullet$}};
		\node[above] at (-.25,0) {\small{$p_1$}};
		
		\node at (1,-1) {\small{$\bullet$}};
		\node[right] at (1,-1) {\small{$a_2$}};
		
		\node at (-1,-1) {\small{$\bullet$}};
		\node[left] at (-1,-1) {\small{$b_2$}};
		
		\node at (-.25,-1) {\small{$\bullet$}};
		\node[right] at (-.25,-1) {\small{$c_2$}};
		
		\node at (-1.75,-.5) {\small{$T_2 =$}};

		\draw[-] (-5.25,0)--(-6,-1);
		\draw[-] (-5.25,0)--(-4,-1);
		\draw[-] (-6,-1)--(-5.25,-1);
		
		\node at (-5.25,0) {\small{$\bullet$}};
		\node[above] at (-5.25,0) {\small{$p_1$}};
		
		\node at (-4,-1) {\small{$\bullet$}};
		\node[right] at (-4,-1) {\small{$a_1$}};
		
		\node at (-6,-1) {\small{$\bullet$}};
		\node[left] at (-6,-1) {\small{$b_1$}};
		
		\node at (-5.25,-1) {\small{$\bullet$}};
		\node[right] at (-5.25,-1) {\small{$c_1$}};
		
		\node at (-6.75,-.5) {\small{$T_1 =$}};
		\node at (-3,-.5) {;};
		\end{tikzpicture}
		\end{center}
	As $\omega(T_1)=\omega(T_2)$, we have
	$$
	\cons(p_1)=\cons(p_2),\text{ and }\ \cons(a_1;b_1,c_1)=\cons(a_2;b_2,c_2).
	$$
	The first equality implies $p_1=p_2$. As $a_1$ and $a_2$ are the largest letters in $\omega_2(T_1)$ respectively $\omega_2(T_2)$ satisfying $a_1\leq p_1$ and $a_2\leq p_2=p_1$, we have that $a_1=a_2$. Furthermore $\omega_2(T_i)$ contains $a_1+b_1=a_2+b_2$ unbarred letters, implying $b_1=b_2$, and $c_1=c_2$ barred letters. Thus indeed $T_1=T_2$.
	
	\smallskip
	\noindent The proof in full of the injectivity of $w$ follows similarly by induction.
	\noindent The first part of the theorem asserts that $w \circ \mathcal{T} = \text{id}_{\mathcal{GT}(n)}$. For $T\in\mathcal{GT}(n)$ we have by the first assertion of the theorem
	$$
	\omega((\mathcal{T}\circ w)(T))=\omega(T),
	$$
	and by injectivity we have $(\mathcal{T}\circ \omega )(T)=T$. Thus indeed $\mathcal{T}=\omega^{-1}$.
	\end{proof}
	\noindent From this theorem, in what follows, we will often denote words of $HW_n$ by their corresponding $C-$tree $T=\mathcal{T}(u)\in\mathcal{GT}(n)$.
	
	\section{Coherent presentations for $Pl^\mathbb{N}(C_n)$ and $Pl(C_n)$}
	\label{sec:coherent presentations}
	
	\noindent In this section we will compute the confluence diagrams of the critical branchings of the presentation \pocol\ of $Pl^\mathbb{N}(C_n)$ to compute the normal form of a $C-$tree $T$ of rank $2$
	
	$$
	
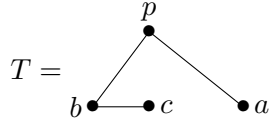
\begin{tikzpicture}
	\draw[-] (-0.25,0)--(-1,-1);
	\draw[-] (-0.25,0)--(1,-1);
	\draw[-] (-1,-1)--(-0.25,-1);
	
	\node at (-.25,0) {\small{$\bullet$}};
	\node[above] at (-.25,0) {\small{$p$}};
	
	\node at (1,-1) {\small{$\bullet$}};
	\node[right] at (1,-1) {\small{$a$}};
	
	\node at (-1,-1) {\small{$\bullet$}};
	\node[left] at (-1,-1) {\small{$b$}};
	
	\node at (-.25,-1) {\small{$\bullet$}};
	\node[right] at (-.25,-1) {\small{$c$}};
	
	\node at (-1.75,-.5) {\small{$T =$}};
	\end{tikzpicture}
	$$
	we need to compute the insertion $\left(\omega_1(T)\leftarrow \omega_2(T)\right)=\left(\cons(p)\leftarrow \cons(a)\cons(p;b,c)\right)$. We know that $(\omega_1(T)\leftarrow \omega_2(T))=[\omega(T)]$, and by Theorem \ref{thm:normalofC}, we have that
	$$
	
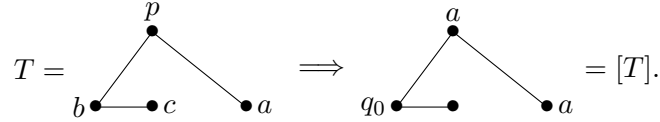
\begin{tikzpicture}
	\draw[-] (-0.25,0)--(-1,-1);
	\draw[-] (-0.25,0)--(1,-1);
	\draw[-] (-1,-1)--(-0.25,-1);
	
	\node at (-.25,0) {\small{$\bullet$}};
	\node[above] at (-.25,0) {\small{$p$}};
	
	\node at (1,-1) {\small{$\bullet$}};
	\node[right] at (1,-1) {\small{$a$}};
	
	\node at (-1,-1) {\small{$\bullet$}};
	\node[left] at (-1,-1) {\small{$b$}};
	
	\node at (-.25,-1) {\small{$\bullet$}};
	\node[right] at (-.25,-1) {\small{$c$}};
	
	\node at (-1.75,-.5) {\small{$T =$}};
	
	\draw[-] (3.75,0)--(3,-1);
	\draw[-] (3.75,0)--(5,-1);
	\draw[-] (3,-1)--(3.75,-1);
	
	\node at (3.75,0) {\small{$\bullet$}};
	\node[above] at (3.75,0) {\small{$a$}};
	
	\node at (5,-1) {\small{$\bullet$}};
	\node[right] at (5,-1) {\small{$a$}};
	
	\node at (3,-1) {\small{$\bullet$}};
	\node[left] at (3,-1) {\small{$q_0$}};
	
	\node at (3.75,-1) {\small{$\bullet$}};
	\node[right] at (3.75,-1) {\small{}};
	
	\node at (2,-.5) {\small{$\Longrightarrow$}};
	
	\node at (6,-.5) {\small{$=[T].$}};
	\end{tikzpicture}
	$$
	with $q_0=p+b-c-a$.
	\noindent Let now $w=tuv\in\adm^*$, and let $w^0$ its highest weight component. By definition of the Kashiwara operators, we know that $w^0=t'u'v'$ for some $t',u',v'\in\adm$. Recall from Section \ref{sec:red sequences} that
	$$
	\text{conf}(w)=\text{conf}(w^0)=(|a(w^0)|,|b(w^0)|),
	$$
	where $a(w^0)=(1,2,1,\dots)$, and $b(w^0)=(2,1,2,\dots)$ are the two reduction strategies for $w^0$. From Theorem \ref{thm:kindamain}, we consider the $C-$tree $T=\mathcal{T}(w^0)\in\mathcal{GT}_3(n)$. It is clear that the reduction strategy $a(T)$ and $b(T)$ corresponds to successive alternating insertions $\omega_1(T)\leftarrow \omega_2(T)$, and $\omega_2(T)\leftarrow \omega_3(T)$, with $a(T)$ starting with $\omega_1(T)\leftarrow \omega_2(T)$, and $b(T)$ starting with $\omega_2(T)\leftarrow \omega_3(T)$. In what follows, we will identify the maximal lengths of each of these sequences $a(T)$ and $b(T)$.
	
	\subsection{Upper bound for $a(T)$}
	
	\begin{thm}\label{thm:main ri len}
		Let $T\in\mathcal{GT}_3(n)$. Then $|a(T)|\leq 4$.
	\end{thm}
	\begin{proof}
		Consider a tree $T\in\mathcal{GT}_3(n)$ in the proof of Theorem \ref{thm:normalofC}. Then the equations \eqref{eq:111},\eqref{eq:222},\eqref{eq:333}, and \eqref{eq:444} for $T\in\mathcal{GT}_3(n)$ take the form
		\begin{align*}
		T'=T\left((\omega_1(T)\leftarrow \omega_2(T))\omega_3(T)\right),\\
		T''=T\left((\omega_1(T')(\omega_2(T')\leftarrow \omega_3(T')))\right),\\
		T'''=T\left((\omega_1(T'')\leftarrow \omega_2(T''))\omega_3(T'')\right),\\
		T''''=T\left((\omega_1(T''')(\omega_2(T''')\leftarrow \omega_3(T''')))\right),
		\end{align*}
		with $T''''$ being normal. Thus we see that $|a(T)|\leq 4$. Which completes the proof of the Theorem.
	\end{proof}	
	\noindent Here we describe each of the $C-$trees in the reduction sequence from $T$ to $[T]$ in $\mathcal{GT}_3(n)$.
	
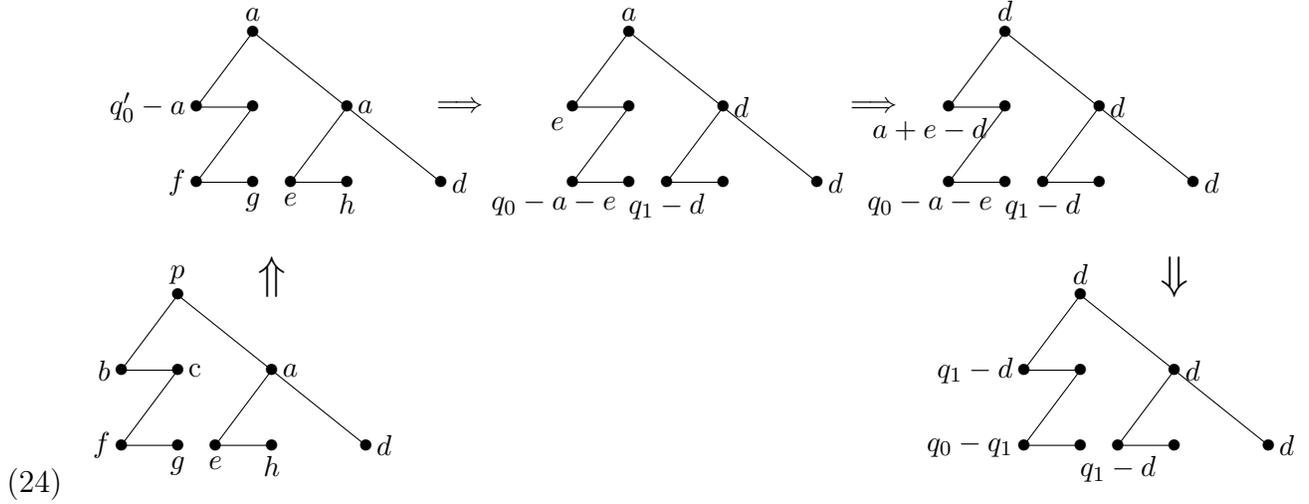
	\begin{equation}\label{eq: gen rk 3 tree}
	\begin{tikzpicture}
	\draw[-] (-5.25,0)--(-6,-1);
	\draw[-] (-6,-1)--(-5.25,-1);
	\draw[-] (-5.25,-1)--(-6,-2);
	\draw[-] (-6,-2)--(-5.25,-2);
	
	\draw[-] (-5.25,0)--(-4,-1);
	\draw[-] (-4,-1)--(-2.75,-2);
	
	\draw[-] (-4,-1)--(-4.75,-2);
	\draw[-] (-4.75,-2)--(-4,-2);
	
	\node at (-5.25,0) {\small{$\bullet$}};
	\node[above] at (-5.25,0) {\small{$a$}};
	
	\node at (-4,-1) {\small{$\bullet$}};
	\node[right] at (-4,-1) {\small{$a$}};
	
	\node at (-6,-1) {\small{$\bullet$}};
	\node[left] at (-6,-1) {\small{$q_0'-a$}};
	
	\node at (-5.25,-1) {\small{$\bullet$}};
	\node[right] at (-5.25,-1) {\small{}};
	
	\node at (-2.75,-2) {\small{$\bullet$}};
	\node[right] at (-2.75,-2) {\small{$d$}};
	
	\node at (-4.75,-2) {\small{$\bullet$}};
	\node[below] at (-4.75,-2) {\small{$e$}};
	
	\node at (-6,-2) {\small{$\bullet$}};
	\node[left] at (-6,-2) {\small{$f$}};
	
	\node at (-5.25,-2) {\small{$\bullet$}};
	\node[below] at (-5.25,-2) {\small{$g$}};

	\node at (-4,-2) {\small{$\bullet$}};
	\node[below] at (-4,-2) {\small{$h$}};
	
	
	\draw[-] (-.25,0)--(-1,-1);
	\draw[-] (-1,-1)--(-.25,-1);
	\draw[-] (-.25,-1)--(-1,-2);
	\draw[-] (-1,-2)--(-.25,-2);
	
	\draw[-] (-.25,0)--(1,-1);
	\draw[-] (1,-1)--(2.25,-2);
	
	\draw[-] (1,-1)--(.25,-2);
	\draw[-] (0.25,-2)--(1,-2);
	
	\node at (-.25,0) {\small{$\bullet$}};
	\node[above] at (-.25,0) {\small{$a$}};
	
	\node at (1,-1) {\small{$\bullet$}};
	\node[right] at (1,-1) {\small{$d$}};
	
	\node at (-1,-1) {\small{$\bullet$}};
	\node[below] at (-1.2,-1) {\small{$e$}};
	
	\node at (-.25,-1) {\small{$\bullet$}};
	\node[right] at (-.25,-1) {\small{}};
	
	\node at (2.25,-2) {\small{$\bullet$}};
	\node[right] at (2.25,-2) {\small{$d$}};
	
	\node at (.25,-2) {\small{$\bullet$}};
	\node[below] at (.25,-2) {\small{$q_1-d$}};
	
	\node at (-1,-2) {\small{$\bullet$}};
	\node[below] at (-1.25,-2) {\small{$q_0-a-e$}};
	
	\node at (-.25,-2) {\small{$\bullet$}};
	\node[below] at (-.25,-2) {\small{}};
	
	\node at (1,-2) {\small{$\bullet$}};
	\node[below] at (1,-2) {\small{}};
	
	\node at (-2.5,-1) {\small{$\Longrightarrow$}};

	
	\draw[-] (4.75,0)--(4,-1);
	\draw[-] (4,-1)--(4.75,-1);
	\draw[-] (4.75,-1)--(4,-2);
	\draw[-] (4,-2)--(4.75,-2);
	
	\draw[-] (4.75,0)--(6,-1);
	\draw[-] (6,-1)--(7.25,-2);
	
	\draw[-] (6,-1)--(5.25,-2);
	\draw[-] (5.25,-2)--(6,-2);
	
	\node at (4.75,0) {\small{$\bullet$}};
	\node[above] at (4.75,0) {\small{$d$}};
	
	\node at (6,-1) {\small{$\bullet$}};
	\node[right] at (6,-1) {\small{$d$}};
	
	\node at (4,-1) {\small{$\bullet$}};
	\node[below] at (3.75,-1) {\small{$a+e-d$}};
	
	\node at (4.75,-1) {\small{$\bullet$}};
	\node[right] at (4.75,-1) {\small{}};
	
	\node at (7.25,-2) {\small{$\bullet$}};
	\node[right] at (7.25,-2) {\small{$d$}};
	
	\node at (5.25,-2) {\small{$\bullet$}};
	\node[below] at (5.25,-2) {\small{$q_1-d$}};
	
	\node at (4,-2) {\small{$\bullet$}};
	\node[below] at (3.75,-2) {\small{$q_0-a-e$}};
	
	\node at (4.75,-2) {\small{$\bullet$}};
	\node[below] at (4.75,-2) {\small{}};
	
	\node at (6,-2) {\small{$\bullet$}};
	\node[below] at (6,-2) {\small{}};
	
	\node at (3,-1) {\small{$\Longrightarrow$}};

	
	\draw[-] (-6.25,-3.5)--(-7,-4.5);
	\draw[-] (-7,-4.5)--(-6.25,-4.5);
	\draw[-] (-6.25,-4.5)--(-7,-5.5);
	\draw[-] (-7,-5.5)--(-6.25,-5.5);
	
	\draw[-] (-6.25,-3.5)--(-5,-4.5);
	\draw[-] (-5,-4.5)--(-3.75,-5.5);
	
	\draw[-] (-5,-4.5)--(-5.75,-5.5);
	\draw[-] (-5.75,-5.5)--(-5,-5.5);
	
	\node at (-6.25,-3.5) {\small{$\bullet$}};
	\node[above] at (-6.25,-3.5) {\small{$p$}};
	
	\node at (-5,-4.5) {\small{$\bullet$}};
	\node[right] at (-5,-4.5) {\small{$a$}};
	
	\node at (-7,-4.5) {\small{$\bullet$}};
	\node[left] at (-7,-4.5) {\small{$b$}};
	
	\node at (-6.25,-4.5) {\small{$\bullet$}};
	\node[right] at (-6.25,-4.5) {\small{c}};
	
	\node at (-3.75,-5.5) {\small{$\bullet$}};
	\node[right] at (-3.75,-5.5) {\small{$d$}};
	
	\node at (-5.75,-5.5) {\small{$\bullet$}};
	\node[below] at (-5.75,-5.5) {\small{$e$}};
	
	\node at (-7,-5.5) {\small{$\bullet$}};
	\node[left] at (-7,-5.5) {\small{$f$}};
	
	\node at (-6.25,-5.5) {\small{$\bullet$}};
	\node[below] at (-6.25,-5.5) {\small{$g$}};
	
	\node at (-5,-5.5) {\small{$\bullet$}};
	\node[below] at (-5,-5.5) {\small{$h$}};
	
	
	\draw[-]   (5.75,-3.5)--(5,-4.5);
	\draw[-] (5,-4.5)--(5.75,-4.5);
	\draw[-] (5.75,-4.5)--(5,-5.5);
	\draw[-] (5,-5.5)--(5.75,-5.5);
	
	\draw[-] (5.75,-3.5)--(7,-4.5);
	\draw[-] (7,-4.5)--(8.25,-5.5);
	
	\draw[-] (7,-4.5)--(6.25,-5.5);
	\draw[-] (6.25,-5.5)--(7,-5.5);
	
	\node at (5.75,-3.5) {\small{$\bullet$}};
	\node[above] at (5.75,-3.5) {\small{$d$}};
	
	\node at (7,-4.5) {\small{$\bullet$}};
	\node[right] at (7,-4.5) {\small{$d$}};
	
	\node at (5,-4.5) {\small{$\bullet$}};
	\node[left] at (5,-4.5) {\small{$q_1-d$}};
	
	\node at (5.75,-4.5) {\small{$\bullet$}};
	\node[right] at (5.75,-4.5) {\small{}};
	
	\node at (8.25,-5.5) {\small{$\bullet$}};
	\node[right] at (8.25,-5.5) {\small{$d$}};
	
	\node at (6.25,-5.5) {\small{$\bullet$}};
	\node[below] at (6.25,-5.5) {\small{$q_1-d$}};
	
	\node at (5,-5.5) {\small{$\bullet$}};
	\node[left] at (5,-5.5) {\small{$q_0-q_1$}};
	
	\node at (5.75,-5.5) {\small{$\bullet$}};
	\node[below] at (5.75,-5.5) {\small{}};
	
	\node at (7,-5.5) {\small{$\bullet$}};
	\node[below] at (7,-5.5) {\small{}};
	
	\node at (-5,-3.25) {\Large{$\Uparrow$}};
	\node at (7,-3.25) {\Large{$\Downarrow$}};
	\end{tikzpicture}
	\end{equation}
	\noindent where $q_0'=p+b-c$.
	\subsection{Upper bound for $b(T)$}
	\noindent We begin with an auxiliary result. This result is in a more general form than we need here.
	
	\begin{prop}\label{prop: h disappears}
		Let $T\in\mathcal{GT}_k(n)$ with $k\geq 1$ and assume that $\omega_k(T)\preceq \omega_{k-1}(T)$. Then $s((k-1)1^-)=0$.
	\end{prop}
	\begin{proof}
		As $\omega_k(T)\preceq \omega_{k-1}(T)$, then $\omega_{k-1}(T)\omega_k(T)$ is standard in $\adm^*$. Set $a=s(k-1)$, $d=s(k)$, $e=s((k-1)1)$, and $h=s((n-1)1^-)$. By the column conditions, we have $d\leq a$.
		
		\smallskip
		\noindent Assume first that $a=0$. Then $d=0$. Note that if $e=0$, then we clearly have $h=0$, as $h\leq a+e=0$. If $e>0$, we have $\omega_k(T)=12\cdots e\cdots$.
		If we had $h>0$, then $\overline{e}\in \omega_k(T)$, but then
		$
		N_e(\omega_k(T))\geq e+1
		$
		which contradicts the admissibility of $\omega_k(T)$. Thus we have $h=0$.
		
		\smallskip
		\noindent Assume now that $a\neq 0$. Again, if $e=0$, by a similar reasoning as in the previous case, we have $h=0$. So suppose that $e>0$. Then the first $d+1$ letters of $\omega_{k-1}(T)$ and $\omega_k(T)$ are respectively $12\cdots d (d+1)$ and $12\cdots d (a+1)$. Since we have $\omega_k(T)\preceq \omega_{k-1}(T)$, we have that in particular $a+1\leq d+1$, which gives $a=d$. But then we have that the first $d+e$ letters of $\omega_k(T)$ are $12\cdots (d+e)$, and if $h\neq 0$, we have $\overline{d+e}\in \omega_k(T)$, which contradicts the admissibility of $\omega_k(T)$. Hence indeed we have $h=0$.
		
		\smallskip
		\noindent This way we have shown that $s((k-1)1^-)=0$ in each case, and thus we have proved the proposition.
	\end{proof}
	\noindent Now we use Theorem \ref{thm:main ri len} and Proposition \ref{prop: h disappears} to find an upper bound for $b(T)$.
	
	\begin{thm}\label{thm:red len le}
		Let $T\in\mathcal{GT}_3(n)$. Then $|b(T)|\leq 3$.
	\end{thm}
	\begin{proof}
		Let
		$$
		\begin{tikzpicture}
		\draw[-] (4,0)--(3.25,-1);
		\draw[-] (3.25,-1)--(4,-1);
		\draw[-] (4,-1)--(3.25,-2);
		\draw[-] (3.25,-2)--(4,-2);
		
		\draw[-] (4,0)--(5.25,-1);
		\draw[-] (5.25,-1)--(6.5,-2);
		
		\draw[-] (5.25,-1)--(4.5,-2);
		\draw[-] (4.5,-2)--(5.25,-2);
		
		\node at (4,0) {\small{$\bullet$}};
		\node[above] at (4,0) {\small{$p$}};
		
		\node at (5.25,-1) {\small{$\bullet$}};
		\node[right] at (5.25,-1) {\small{$a$}};
		
		\node at (3.25,-1) {\small{$\bullet$}};
		\node[left] at (3.25,-1) {\small{$b$}};
		
		\node at (4,-1) {\small{$\bullet$}};
		\node[above] at (4,-1) {\small{$c$}};
		
		\node at (6.5,-2) {\small{$\bullet$}};
		\node[right] at (6.5,-2) {\small{$d$}};
		
		\node at (4.5,-2) {\small{$\bullet$}};
		\node[below] at (4.5,-2) {\small{$e$}};
		
		\node at (3.25,-2) {\small{$\bullet$}};
		\node[left] at (3.25,-2) {\small{$f$}};
		
		\node at (4,-2) {\small{$\bullet$}};
		\node[above] at (4,-2) {\small{$g$}};

		\node at (5.25,-2) {\small{$\bullet$}};
		\node[above] at (5.25,-2) {\small{$h$}};
		\node at (2.25,-1) {\small{$T=$}};
		\end{tikzpicture}
		$$		
		Then by Proposition \ref{prop: h disappears}, we have that
		$$
		\begin{tikzpicture}
		\draw[-] (-0.25,0)--(-1,-1);
		\draw[-] (-1,-1)--(-0.25,-1);
		\draw[-] (-0.25,-1)--(-1,-2);
		\draw[-] (-1,-2)--(-0.25,-2);
		
		\draw[-] (-0.25,0)--(1,-1);
		\draw[-] (1,-1)--(2.25,-2);
		
		\draw[-] (1,-1)--(0.25,-2);
		\draw[-] (0.25,-2)--(1,-2);
		
		\node at (-.25,0) {\small{$\bullet$}};
		\node[above] at (-.25,0) {\small{$p$}};
		
		\node at (1,-1) {\small{$\bullet$}};
		\node[right] at (1,-1) {\small{$d$}};
		
		\node at (-1,-1) {\small{$\bullet$}};
		\node[left] at (-1,-1) {\small{$b_1$}};
		
		\node at (-.25,-1) {\small{$\bullet$}};
		\node[right] at (-.25,-1) {\small{$c_1$}};
		
		\node at (2.25,-2) {\small{$\bullet$}};
		\node[right] at (2.25,-2) {\small{$d$}};
		
		\node at (.25,-2) {\small{$\bullet$}};
		\node[below] at (.25,-2) {\small{$e_1$}};
		
		\node at (-1,-2) {\small{$\bullet$}};
		\node[left] at (-1,-2) {\small{$f_1$}};
		
		\node at (-.25,-2) {\small{$\bullet$}};
		\node[below] at (-.25,-2) {\small{$g_1$}};

		\node at (1,-2) {\small{$\bullet$}};
		\node[below] at (1,-2) {\small{}};
		
		\node at (-4.6,-1) {\small{$T'=\mathcal{T}\left(\omega_1(T)(\omega_2(T)\leftarrow \omega_3(T))\right) =$}};
		\end{tikzpicture}
		$$
		Since $[T]=[T']$, by Corollary \ref{cor:qi are preserved} we have that $q_i(T)=q_i(T')$. More precisely we have
		$
		a+e-h=q_1(T)=q_1(T')=d+e_1
		$
		which gives us $e_1=a+e-h-d=q_1-d$.
		Next we have
		$$
		\begin{tikzpicture}
		\draw[-] (-0.25,0)--(-1,-1);
		\draw[-] (-1,-1)--(-0.25,-1);
		\draw[-] (-0.25,-1)--(-1,-2);
		\draw[-] (-1,-2)--(-0.25,-2);
		
		\draw[-] (-0.25,0)--(1,-1);
		\draw[-] (1,-1)--(2.25,-2);
		
		\draw[-] (1,-1)--(0.25,-2);
		\draw[-] (0.25,-2)--(1,-2);
		
		\node at (-.25,0) {\small{$\bullet$}};
		\node[above] at (-.25,0) {\small{$d$}};
		
		\node at (1,-1) {\small{$\bullet$}};
		\node[right] at (1,-1) {\small{$d$}};
		
		\node at (-1,-1) {\small{$\bullet$}};
		\node[below left] at (-1,-1) {\small{$p+b_1-c_1-d$}};
		
		\node at (-.25,-1) {\small{$\bullet$}};
		\node[right] at (-.25,-1) {\small{}};
		
		\node at (2.25,-2) {\small{$\bullet$}};
		\node[right] at (2.25,-2) {\small{$d$}};
		
		\node at (.25,-2) {\small{$\bullet$}};
		\node[below] at (.25,-2) {\small{$q_1-d$}};
		
		\node at (-1,-2) {\small{$\bullet$}};
		\node[left] at (-1,-2) {\small{$f_1$}};
		
		\node at (-.25,-2) {\small{$\bullet$}};
		\node[above] at (-.25,-2) {\small{$g_1$}};

		\node at (1,-2) {\small{$\bullet$}};
		\node[below] at (1,-2) {\small{}};
		
		\node at (5,-1) {\small{$=T''=\mathcal{T}\left((\omega_1(T')\leftarrow \omega_2(T'))\omega_2(T')\right)$}};
		
		\end{tikzpicture}
		$$
		We note that $T''$ is of the same form as the fourth $C-$tree in \eqref{eq: gen rk 3 tree}, hence the final computation there we obtain
		$$
		\begin{tikzpicture}
		\draw[-] (-0.25,0)--(-1,-1);
		\draw[-] (-1,-1)--(-0.25,-1);
		\draw[-] (-0.25,-1)--(-1,-2);
		\draw[-] (-1,-2)--(-0.25,-2);
		
		\draw[-] (-0.25,0)--(1,-1);
		\draw[-] (1,-1)--(2.25,-2);
		
		\draw[-] (1,-1)--(0.25,-2);
		\draw[-] (0.25,-2)--(1,-2);
		
		\node at (-.25,0) {\small{$\bullet$}};
		\node[above] at (-.25,0) {\small{$d$}};
		
		\node at (1,-1) {\small{$\bullet$}};
		\node[right] at (1,-1) {\small{$d$}};
		
		\node at (-1,-1) {\small{$\bullet$}};
		\node[below left] at (-1,-1) {\small{$q_1-d$}};
		
		\node at (-.25,-1) {\small{$\bullet$}};
		\node[right] at (-.25,-1) {\small{}};
		
		\node at (2.25,-2) {\small{$\bullet$}};
		\node[right] at (2.25,-2) {\small{$d$}};
		
		\node at (.25,-2) {\small{$\bullet$}};
		\node[below] at (.25,-2) {\small{$q_1-d$}};
		
		\node at (-1,-2) {\small{$\bullet$}};
		\node[left] at (-1,-2) {\small{$q_0-d-e_1$}};
		
		\node at (-.25,-2) {\small{$\bullet$}};
		\node[below] at (-.25,-2) {\small{}};

		\node at (1,-2) {\small{$\bullet$}};
		\node[below] at (1,-2) {\small{}};
		
		\node at (5,-1) {\small{$=T'''=\mathcal{T}((\omega_1(T)(\omega_2(T)\leftarrow \omega_3(T)))$}};
		\end{tikzpicture}
		$$
		and since $q_1=d+e_1$, we see that $T'''$ is indeed standard thus we obtain $|b(T)|\leq 3$, which is what we wanted to show.
	\end{proof}
	
	\begin{exo}
		Theorems \ref{thm:main ri len} and \ref{thm:red len le} give only upper bounds on $\text{conf}(T)$. Here we show that these are upper bounds are optimal.
		
		\smallskip
		\noindent Consider the $C-$tree $T\in\mathcal{GT}_3(n)$
		$$
		\begin{tikzpicture}
		\draw[-] (4,0)--(3.25,-1);
		\draw[-] (3.25,-1)--(4,-1);
		\draw[-] (4,-1)--(3.25,-2);
		\draw[-] (3.25,-2)--(4,-2);
		
		\draw[-] (4,0)--(5.25,-1);
		\draw[-] (5.25,-1)--(6.5,-2);
		
		\draw[-] (5.25,-1)--(4.5,-2);
		\draw[-] (4.5,-2)--(5.25,-2);
		
		\node at (4,0) {\small{$\bullet$}};
		\node[above] at (4,0) {\small{$2$}};
		
		\node at (5.25,-1) {\small{$\bullet$}};
		\node[right] at (5.25,-1) {\small{$1$}};
		
		\node at (3.25,-1) {\small{$\bullet$}};
		\node[left] at (3.25,-1) {\small{}};
		
		\node at (4,-1) {\small{$\bullet$}};
		\node[above] at (4,-1) {\small{}};
		
		\node at (6.5,-2) {\small{$\bullet$}};
		\node[right] at (6.5,-2) {\small{}};
		
		\node at (4.5,-2) {\small{$\bullet$}};
		\node[below] at (4.5,-2) {\small{$1$}};
		
		\node at (3.25,-2) {\small{$\bullet$}};
		\node[left] at (3.25,-2) {\small{}};
		
		\node at (4,-2) {\small{$\bullet$}};
		\node[above] at (4,-2) {\small{}};

		\node at (5.25,-2) {\small{$\bullet$}};
		\node[below] at (5.25,-2) {\small{$1$}};
		\node at (2.25,-1) {\small{$T=$}};
		\end{tikzpicture}
		$$
		with word $\omega(T)=tuv=1212\overline{2}$, where $t=12$, $u=1$, and $v=2\overline{2}$.  We then have
		
		\begin{align*}
		T'=(12\leftarrow 1)2\overline{2}=(1)(12)(2\overline{2})\\
		T''=1(12\leftarrow 2\overline{2})=(1)(2)(1)\\
		T'''=(1\leftarrow 2)1=(\epsilon)(12)(1)\\
		T''''=\epsilon (12\leftarrow 1)=(\epsilon) (1)(12)
		\end{align*}
		thus indeed $|a(T)|=4$.
		
		\noindent To compute $|b(T)|$ we have
		\begin{align*}
		R'=12(1\leftarrow 2\overline{2})=(12)(\epsilon)(1)\\
		R''=(12\leftarrow \epsilon)1=(\epsilon)(12)(1)\\
		R'''=\epsilon (12\leftarrow 1)=(\epsilon)(1)(12)
		\end{align*}
		so that we have $|b(T)|=3$. Thus we have $\text{conf}(T)=(3,4)$.
		
		\label{ex:importan examplos}
	\end{exo}
	
	\noindent We are now ready to formalize our results in terms of presentations. Squier's theorem from \cite{squier1994finiteness} asserts given a convergent presentation, a family of generating confluences of \pocol\ forms a generating set for the $3-$cells of its coherent extension. In the case of the presentation \pocol\ of $Pl^\mathbb{N}(C_n)$, the family of generating confluences is those with source in $\{w=tuv\in \adm^*\ |\ v\npreceq u\npreceq t\}$, namely the critical branchings. In Theorems \ref{thm:main ri len} and \ref{thm:red len le} we have found upper bounds for the two reduction sequences of these critical branchings, and Example \ref{ex:importan examplos} we see that these can in fact be reached. Thus we have the following.
	
	\begin{thm}\label{thm:lastbieq}
		The generating $3-$cells of the coherent presentation \pocol\ of $Pl^\mathbb{N}(C_n)$ are of the form
		\begin{equation}\label{eq:confidag end}
		{{\xymatrix{ & t'u'v \ar@{=>}[r] & t'u''v' \ar@{=>}[r] & t''u'''v' \ar@{=>}[rd]\\ tuv \ar@{=>}[ru]^{\alpha_{tu}v} \ar@{=>}[rd]_{t\alpha_{uv}}&  &  &  & t_0u_0v_0 \\ & t u_1 v_1 \ar@{=>}[rr] & &t_1u_2v_1 \ar@{=>}[ur]}}} 
		\end{equation}
		where we allow for some of the arrows to be the identity.
	\end{thm}
	
	\noindent Thus far we have worked with $Pl^\mathbb{N}(C_n)$ and its convergent presentation \pocol, whose coherent extension is made explicit in Theorem \ref{thm:lastbieq}. We now use this discussion to obtain results about $Pl(C_n)$ and its convergent presentation \pohag.
	
	\smallskip
	\noindent Let $w=tuv\in\textbf{ACol}_\bullet$ the source of a critical branching, i.e. $v\npreceq u\npreceq t$. Let $s_1$ and $s_2$ respectively be the leftmost and rightmost reduction sequences of $w=tuv\in\textbf{ACol}$. From the proof of Theorem \ref{thm:conv poly}, we have reduction sequences $p(s_1)$ and $p(s_2)$ of $w\in\textbf{ACol}_\bullet$. By \eqref{eq:bound lengths} in that proof, and by Theorems \ref{thm:main ri len} and \ref{thm:red len le} we obtain
	\begin{align*}
	|p(s_1)|_\bullet \leq |s_1| \leq 4,\\
	|p(s_2)|_\bullet \leq |s_2| \leq 3.
	\end{align*}
	
	\noindent Note that the leftmost reduction sequence of $w=1212\overline{2}$ in Example \ref{ex:importan examplos} involves no $2-$cell with $\epsilon$ in it, hence the corresponding reduction sequence in $\textbf{ACol}_\bullet$ is also of length $4$.
	
	\begin{exo}
		Let $t=1$,$u=23$,$v=2$ in $\bull$. Clearly we have $v\npreceq u\npreceq t$, and the rightmost reduction sequence of $w=tuv=321$ in $\textbf{ACol}_\bullet$ is as follows
		$$
		s_2: (1)(23)(2)\stackrel{2}{\Longrightarrow} (1)(2)(23)\stackrel{1}{\Longrightarrow}(12)(23)\stackrel{2}{\Longrightarrow}(2)(123),
		$$
		where the label on $\Longrightarrow$ indicates the location in the sequence where we apply a $2-$cell. In particular, we see that $|s|=3$.
	\end{exo}
	\noindent This way, we obtain the following.
	
	\begin{cor}\label{cor:lastbieq}
		The generating $3-$cells of the coherent presentation \pohag\ of $Pl(C_n)$ are of the form
		\begin{equation}\label{eq:confidag end}
		{{\xymatrix{ & t'u'v \ar@{=>}[r] & t'u''v' \ar@{=>}[r] & t''u'''v' \ar@{=>}[rd]\\ tuv \ar@{=>}[ru]^{\alpha_{tu}v} \ar@{=>}[rd]_{t\alpha_{uv}}&  &  &  & t_0u_0v_0 \\ & t u_1 v_1 \ar@{=>}[rr] & &t_1u_2v_1 \ar@{=>}[ur]}}} 
		\end{equation}
		where we allow for some of the arrows to be the identity.
	\end{cor}
\begin{small}
	\addcontentsline{toc}{section}{References}
	\bibliographystyle{plain}
	\bibliography{main}
\end{small}
\end{document}